\documentclass[a4paper,10pt]{amsart}
\usepackage{amsmath,amsfonts,amssymb,amsthm}
\usepackage{color,hyperref}

\theoremstyle{plain}
\newtheorem{theorem}{\bf Theorem}[section]

\newtheorem{lemma}[theorem]{\bf Lemma}

\newtheorem{thm}{Theorem}

\theoremstyle{definition}

\newtheorem{definition}[theorem]{\bf Definition}

\newcommand{\N}{\mathbb N}
\newcommand{\Z}{\mathbb Z}

\renewcommand{\t}{\, | \,}
\newcommand{\und}{\;\mbox{ and }\;}

\newcommand{\la}{\langle}
\newcommand{\ra}{\rangle}

     \DeclareMathOperator{\ord}{ord}
\DeclareMathOperator{\lcm}{lcm}     
 
\DeclareMathOperator{\supp}{supp}

\newcommand{\bdot}{\boldsymbol{\cdot}}

\newcommand{\blfloor}{\big\lfloor}
\newcommand{\brfloor}{\big\rfloor}

\numberwithin{equation}{section}

\begin{document}

\title[Erd{\H{o}}s-Ginzburg-Ziv constant]{On Erd{\H{o}}s-Ginzburg-Ziv inverse theorems for \\ Dihedral and Dicyclic groups}

\author{Jun Seok Oh and Qinghai Zhong}
\address{Institute for Mathematics and Scientific Computing \\ University of Graz, NAWI Graz \\ Heinrichstra{\ss}e 36 \\ 8010 Graz, Austria}
\email{junseok.oh@uni-graz.at, qinghai.zhong@uni-graz.at}
\urladdr{https://imsc.uni-graz.at/zhong/}

\subjclass[2010]{20D60, 11B75, 11P70}
\keywords{product-one sequences, Erd{\H{o}}s-Ginzburg-Ziv constant, Dihedral groups, Dicyclic groups}

\thanks{This work was supported by the Austrian Science Fund FWF, W1230 Doctoral Program ``Discrete Mathematics" and Project No. P28864--N35.}

\begin{abstract}
Let $G$ be a finite group and $\exp (G) = \lcm \{ \ord(g) \t g \in G \}$. A finite unordered sequence of terms from $G$, where repetition is allowed, is a product-one sequence if its terms can be ordered such that their product equals the identity element of $G$. We denote by $\mathsf s (G)$ (or $\mathsf E (G)$ respectively) the smallest integer $\ell$ such that every sequence of length at least $\ell$ has a product-one subsequence of length $\exp (G)$ (or $|G|$ respectively). In this paper, we provide the exact values of $\mathsf s (G)$ and $\mathsf E (G)$ for Dihedral and Dicyclic groups and we provide explicit characterizations of all sequences of length $\mathsf s (G) - 1$ (or $\mathsf E (G) - 1$ respectively) having no product-one subsequence of length $\exp (G)$ (or $|G|$ respectively).
\end{abstract}

\maketitle

%%%%%%%%%%%%%%%%%%%%%%%%%%%%%%%%%%%%%%%%%%%%%%%%%%%%%%%%%%%%%%%%%%%%%%%%%%%%%%%%%%%%%%%%%%%%%%%%%%%%%%%%%%%%%%%%%%%%%%%%%%%%%%%%%%%%%%%%%%%%%%%%%%%%%%%%%%%%%%%%%%%%%%%%%%%
%%%%%%%%%%%%%%%										%%%%%%%%%%%%%%%										%%%%%%%%%%%%%%%										%%%%%%%%%%%%%%%
%%%%%%%%%%%%%%%%%%%%%%%%%%%%%%%%%%%%%%%%%%%%%%%%%%%%%%%%%%%%%%%%%%%%%%%%%%%%%%%%%%%%%%%%%%%%%%%%%%%%%%%%%%%%%%%%%%%%%%%%%%%%%%%%%%%%%%%%%%%%%%%%%%%%%%%%%%%%%%%%%%%%%%%%%%%

\bigskip
\section{Introduction} \label{1}
\bigskip

Let $G$ be a finite group and $\exp (G) = \lcm \{ \ord(g) \t g \in G \}$. By a sequence $S$ over $G$, we mean a finite sequence of terms from $G$ which is unordered, repetition of terms allowed. We say that $S$ is a product-one sequence if its terms can be ordered so that their product equals the identity element of $G$. The {\it small Davenport constant $\mathsf d (G)$} is the maximal integer $\ell$ such that there is a sequence of length $\ell$ which has no non-trivial product-one subsequence. We denote by $\mathsf s (G)$ (or $\mathsf E (G)$ respectively) the smallest integer $\ell$ such that every sequence of length at least $\ell$ has a product-one subsequence of length $\exp (G)$ (or $|G|$ respectively). When $G$ is cyclic, we have that $\exp (G) = |G|$ and $\mathsf s (G) = \mathsf E (G) = \mathsf d (G) + |G|$, which is due to Erd{\H{o}}s, Ginzburg, and Ziv in 1961, whence $\mathsf s (G)$ is called the Erd{\H{o}}s-Ginzburg-Ziv constant of $G$.

Both invariants, $\mathsf s (G)$ and $\mathsf E (G)$, found wide attention for finite abelian groups (see \cite{Ga-Ge06,Ge-HK06,Ge09a,Gr13a} for surveys). Caro (\cite{Ca96a}) and Gao (\cite{Ga95c}) proved independently that $\mathsf E (G) = \mathsf d (G) + |G|$ (for a short proof see \cite[Proposition 5.7.9]{Ge-HK06}). This result has seen far reaching generalizations  (\cite[Chapter 16]{Gr13a}). Much less is known for $\mathsf s (G)$. If $G \cong C_{n_1} \oplus C_{n_2}$ with $1 \le n_1 \mid n_2$, then $\mathsf s (G) = 2n_1+2n_2-3$ (\cite[Theorem 5.8.3]{Ge-HK06}) but for groups of higher rank the precise value of $\mathsf s (G)$ is known only in very special cases   (\cite{Ed-El-Ge-Ku-Ra07,Fa-Ga-Zh11a, Gi-Sc19a,Gi-Sc19b}).

The study of sequences and associated invariants for non-abelian groups dates back to the 1970s (e.g., \cite{Ol-Wh77}), and fresh impetus came from applications in factorization theory and invariant theory (\cite{Cz-Do14a, Ge-Gr13a, Gr13b, Cz-Do13c, Cz-Do15a, Cz-Do-Sz17, Oh18a, Oh19a}). Among others, the formula $\mathsf E (G) = \mathsf d (G) + |G|$ turned out to hold true for various non-abelian groups (see \cite{Ga-Li10b, Ha15a, Ha-Zh19}). Gao and Lu (\cite{Ga-Lu08a}) verified such formula for the dihedral groups of order $2n$, where $n \ge 23$, and Bass (\cite{Ba07}) improved the result to all dihedral and dicyclic groups.

\smallskip
\begin{thm} \label{A}~
If $G$ is a dihedral group of order $2n$ for $n \ge 3$ $($or a dicyclic group of order $4n$ for $n \ge 2$ respectively$)$, then $\mathsf E (G) = \mathsf d (G) + |G| = 3n$ $($or $\mathsf E (G) = \mathsf d (G) + |G| = 6n$ respectively$)$.
\end{thm}

In some earlier papers (\cite{Cz-Do-Ge16,Ha-Zh19}), authors also studied an invariant $\mathsf s' (G)$ defined as the smallest integer $\ell$ such that every sequence over $G$ of length at least $\ell$ has a product-one subsequence of length $\max \{ \ord (g) \t g \in G \}$. If $G$ is nilpotent (in particular, if $G$ is abelian), then $\exp (G) = \max \{\ord (g) \t g \in G \}$, whence $\mathsf s (G) = \mathsf s' (G)$. Similarly, if $G$ is a dihedral group of order $2n$ or a dicyclic group of order $4n$ for even $n$, then we also have $\exp (G) = \max \{\ord (g) \t g \in G \}$. If $G$ is a dihedral group of order $2n$ with $n$ odd, then there are arbitrarily long sequences over $G$ having no product-one subsequence of length $\max \{\ord (g) \t g \in G \}$ (see the discussion after Definition \ref{2.1}). In the present paper, we provide the precise value of $\mathsf s (G)$ for the dihedral and dicyclic groups.

\smallskip
\begin{theorem} \label{1.1}~
Let $n \in \N$.
\begin{enumerate}
\item If $G$ is a dihedral group of order $2n$ for $n \ge 3$, then $\mathsf s (G) = \left\{ \begin{array}{ll}
                                                                                              3n & \hbox{if n is odd} \\
                                                                                              2n & \hbox{if n is even} \,.
                                                                                            \end{array} \right.$

\smallskip
\item If $G$ is a dicyclic group of order $4n$ for $n \ge 2$, then $\mathsf s (G) = \left\{ \begin{array}{ll}
                                                                                              6n & \hbox{if n is odd} \\
                                                                                              4n & \hbox{if n is even} \,.
                                                                                            \end{array} \right.$
\end{enumerate}
\end{theorem}

After discussing the direct problem, which asks for the precise value of group invariants, we consider the associated inverse problem, which asks for the structure of extremal sequences.
Structural results characterizing which sequences achieve equality are rare. Even in the abelian case, little is known precisely outside of groups of rank at most $2$. Among others, we refer to \cite{Br-Ri18a,Oh-Zh19} for very recent works associated the Davenport constants for specific non-abelian groups. In the present paper, we focus on the inverse problem associated to the Erd{\H{o}}s-Ginzburg-Ziv constant, and we prove the following results.

\smallskip
\begin{theorem} \label{1.2}~
Let $n \in \N_{\ge 3}$, let $G$ be a dihedral group of order $2n$, and let $S \in \mathcal F (G)$ be a sequence.
\begin{enumerate}
\item The following statements are equivalent{\rm \,:}
      \begin{enumerate}
      \item[(a)] $n \ge 4$ is even, $|S| = \mathsf s (G) - 1$, and $S$ has no product-one subsequence of length $\exp (G) = n$.

      \item[(b)] There exist $\alpha, \tau \in G$ such that $G = \la \alpha, \tau \t \alpha^{n} = \tau^{2} = 1_G \text{ and } \tau \alpha = \alpha^{-1}\tau \ra$ and $S = (\alpha^{r_1})^{[n-1]} \bdot (\alpha^{r_2})^{[n-1]} \bdot \alpha^{r_3}\tau$, where $r_1, r_2, r_3 \in [0,n-1]$ with $\gcd(r_1 - r_2, n) = 1$.
      \end{enumerate}

\item The following statements are equivalent{\rm \,:}
      \begin{enumerate}
      \item[(a)] $|S| = \mathsf E (G) - 1$ and $S$ has no product-one subsequence of length $|G|$.

      \item[(b)] There exist $\alpha, \tau \in G$ such that $G = \la \alpha, \tau \t \alpha^{n} = \tau^{2} = 1_G \text{ and } \tau \alpha = \alpha^{-1}\tau \ra$ and $S$ has one of the following forms{\rm \,:}
      \begin{enumerate}
      \item[(1)] $n \ge 4$ and $S = (\alpha^{r_1})^{[2n-1]} \bdot (\alpha^{r_2})^{[n-1]} \bdot \alpha^{r_3}\tau$, where $r_1, r_2, r_3 \in [0,n-1]$ with $\gcd(r_1 - r_2, n) = 1$.

      \item[(2)] $n = 3$, and either $S = 1^{[5]}_G \bdot \tau \bdot \alpha\tau \bdot \alpha^{2}\tau$ or $S = (\alpha^{t_1})^{[5]} \bdot (\alpha^{t_2})^{[2]} \bdot \alpha^{t_3}\tau$, where $t_1, t_2, t_3 \in [0,2]$ with $\gcd(r_1 - r_2, 3) = 1$.
      \end{enumerate}
      \end{enumerate}
\end{enumerate}
\end{theorem}

\begin{theorem} \label{1.3}~
Let $n \in \N_{\ge 2}$, let $G$ be a dicyclic group of order $4n$, and let $S \in \mathcal F (G)$ be a sequence.
\begin{enumerate}
\item The following statements are equivalent{\rm \,:}
      \begin{enumerate}
      \item[(a)] $n \ge 2$ is even, $|S| = \mathsf s (G) - 1$, and $S$ has no product-one subsequence of length $\exp (G) = 2n$.

      \item[(b)] There exist $\alpha, \tau \in G$ such that $G = \la \alpha, \tau \t \alpha^{2n} = 1_G, \tau^{2} = \alpha^{n}, \text{and } \tau\alpha = \alpha^{-1}\tau \ra$ and $S$ has one of the following forms{\rm \,:}
      \begin{itemize}
      \item[(1)] $n \ge 4$ and $S = (\alpha^{r_1})^{[2n-1]} \bdot (\alpha^{r_2})^{[2n-1]} \bdot \alpha^{r_3}\tau$, where $r_1, r_2, r_3 \in [0,2n-1]$ with $\gcd(r_1 - r_2, 2n) = 1$.

      \item[(2)] $n = 2$ and $S \in \big\{ (\alpha^{t_1})^{[3]} \bdot (\alpha^{t_2})^{[3]} \bdot \alpha^{t_3}\tau, \,\, (\alpha^{t_1})^{[3]} \bdot (\alpha^{t_3}\tau)^{[3]} \bdot \alpha^{t_2}, \,\, (\alpha^{t_1})^{[3]} \bdot (\alpha^{t_3}\tau)^{[3]} \bdot \alpha^{t_4}\tau \big\}$, where $t_1, t_2, t_3, t_4 \in [0,3]$ such that $t_1$ is even, $t_2$ is odd, and $t_3 \not\equiv t_4 \pmod{2}$.
      \end{itemize}
      \end{enumerate}

\item The following statements are equivalent{\rm \,:}
      \begin{enumerate}
      \item[(a)] $|S| = \mathsf E (G) - 1$ and $S$ has no product-one subsequence of length $|G|$.

      \item[(b)] There exist $\alpha, \tau \in G$ such that $G = \la \alpha, \tau \t \alpha^{2n} = 1_G, \tau^{2} = \alpha^{n}, \text{and } \tau\alpha = \alpha^{-1}\tau \ra$ and $S$ has one of the following forms{\rm \,:}
      \begin{itemize}
      \item[(1)] $n \ge 3$ and $S = (\alpha^{r_1})^{[4n-1]} \bdot (\alpha^{r_2})^{[2n-1]} \bdot \alpha^{r_3}\tau$, where $r_1, r_2, r_3 \in [0,2n-1]$ with $\gcd(r_1 - r_2, 2n) = 1$.

      \item[(2)] $n = 2$ and $S \in \Big\{ \big( (\alpha^{t_1})^{[3]} \bdot (\alpha^{t_2})^{[3]} \bdot \alpha^{t_3}\tau \big) \bdot S_0, \, \big( (\alpha^{t_1})^{[3]} \bdot (\alpha^{t_3}\tau)^{[3]} \bdot \alpha^{t_2} \big) \bdot S_1, \, \big( (\alpha^{t_1})^{[3]} \bdot (\alpha^{t_3}\tau)^{[3]} \bdot \alpha^{t_4}\tau \big) \bdot S_1 \Big\}$, where $S_0 \in \big\{ (\alpha^{t_1})^{[4]}, \, (\alpha^{t_2})^{[4]} \big\}$, $S_1 \in \big\{ (\alpha^{t_1})^{[4]}, \, (\alpha^{t_3}\tau)^{[4]} \big\}$, and $t_1, t_2, t_3, t_4 \in [0,3]$ such that $t_1$ is even, $t_2$ is odd, and $t_3 \not\equiv t_4 \pmod{2}$.
      \end{itemize}
      \end{enumerate}
\end{enumerate}
\end{theorem}

%%%%%%%%%%%%%%%%%%%%%%%%%%%%%%%%%%%%%%%%%%%%%%%%%%%%%%%%%%%%%%%%%%%%%%%%%%%%%%%%%%%%%%%%%%%%%%%%%%%%%%%%%%%%%%%%%%%%%%%%%%%%%%%%%%%%%%%%%%%%%%%%%%%%%%%%%%%%%%%%%%%%%%%%%%%
%%%%%%%%%%%%%%%										%%%%%%%%%%%%%%%										%%%%%%%%%%%%%%%										%%%%%%%%%%%%%%%
%%%%%%%%%%%%%%%%%%%%%%%%%%%%%%%%%%%%%%%%%%%%%%%%%%%%%%%%%%%%%%%%%%%%%%%%%%%%%%%%%%%%%%%%%%%%%%%%%%%%%%%%%%%%%%%%%%%%%%%%%%%%%%%%%%%%%%%%%%%%%%%%%%%%%%%%%%%%%%%%%%%%%%%%%%%

\bigskip
\section{Preliminaries} \label{2}
\bigskip

Much of the following notation can be found in \cite{Oh-Zh19} and is repeated here for the convenience of the reader. We denote by $\N$ the set of positive integers and we set $\N_0 = \N \cup \{ 0 \}$. For each $k \in \N$, we also denote by $\N_{\ge k}$ the set of positive integers greater than or equal to $k$. For integers $a, b \in \Z$, $[a,b] = \{x \in \Z \mid a \le x \le b\}$ is the discrete interval. 
%For two sets $A$ and $B$, we denote by $A \cupdot B$ the disjoint union of $A$ and $B$.

\smallskip
\noindent
{\bf Sequences over groups.} Let $G$ be a multiplicatively written finite group with identity element $1_G$ and let $G_0 \subset G$ be a subset. For an element $g \in G$, we denote by $\ord(g) \in \N$ the order of $g$, by $\exp (G) = \lcm \{ \ord(g) \t g \in G \}$ the exponent of $G$,  and by $\la G_0 \ra \subset G$ the subgroup generated by $G_0$.

The elements of the free abelian monoid $\mathcal F (G_0)$ will be called  {\it sequences} over $G_0$.  This terminology goes back to Combinatorial Number Theory. Indeed, a sequence over $G_0$ can be viewed as a finite unordered sequence of terms from $G_0$, where the repetition of elements is allowed. We briefly discuss our notation which follows  the monograph \cite[Chapter 10.1]{Gr13a}. In order to avoid confusion between multiplication in $G$ and multiplication in $\mathcal F (G_0)$, we denote multiplication in $\mathcal F (G_0)$ by the boldsymbol $\bdot$ and we use brackets for all exponentiation in $\mathcal F (G_0)$. In particular, a sequence $S \in \mathcal F (G_0)$ has the form
\begin{equation} \label{basic}
S \, = \, g_1 \bdot \ldots \bdot g_{\ell} \, = \, {\small \prod}^{\bullet}_{i \in [1,\ell]} \, g_i \, \in \, \mathcal F (G_0),
\end{equation}
where $g_1, \ldots, g_{\ell} \in G_0$ are the terms of $S$. For $g \in G_0$,
\begin{itemize}
\item $\mathsf v_g (S) = | \{ i \in [1,\ell] \t g_i = g \} |$ denotes the {\it multiplicity} of $g$ in $S$,

\smallskip
\item $\supp(S) = \{ g \in G_0 \t \mathsf v_{g} (S) > 0 \}$ denotes the {\it support} of $S$, and

\smallskip
\item $\mathsf h (S) = \max \{ \mathsf v_g (S) \t g \in G_0 \}$ denotes the {\it maximal multiplicity} of $S$.
\end{itemize}
A {\it subsequence} $T$ of $S$ is a divisor of $S$ in $\mathcal F (G_0)$ and we write $T \t S$. For a subset $H \subset G_0$, we denote by $S_H$ the subsequence of $S$ consisting of all terms from $H$. Furthermore, $T \t S$ if and only if $\mathsf v_g (T) \le \mathsf v_g (S)$ for all $g \in G_0$, and in such case, $S \bdot T^{[-1]}$ denotes the subsequence of $S$ obtained by removing the terms of $T$ from $S$ so that $\mathsf v_g \big( S \bdot T^{[-1]} \big) = \mathsf v_g (S) - \mathsf v_g (T)$ for all $g \in G_0$. On the other hand, we set $S^{-1} = g^{-1}_1 \bdot \ldots \bdot g^{-1}_{\ell}$ to be the sequence obtained by taking elementwise inverse from $S$.

Moreover, if $S_1, S_2 \in \mathcal F (G_0)$ and $g_1, g_2 \in G_0$, then $S_1 \bdot S_2 \in \mathcal F (G_0)$ has length $|S_1|+|S_2|$, \ $S_1 \bdot g_1 \in \mathcal F (G_0)$ has length $|S_1|+1$, \ $g_1g_2 \in G$ is an element of $G$, but $g_1 \bdot g_2 \in \mathcal F (G_0)$ is a sequence of length $2$. If $g \in G_0$, $T \in \mathcal F (G_0)$, and $k \in \N_0$, then
\[
g^{[k]}=\underset{k}{\underbrace{g\bdot\ldots\bdot g}}\in \mathcal F (G_0) \quad \text{and} \quad T^{[k]}=\underset{k}{\underbrace{T\bdot\ldots\bdot T}}\in \mathcal F (G_0) \,.
\]
Let $S \in \mathcal F (G_0)$ be a sequence as in \eqref{basic}. Then we denote by
\[
  \pi(S) \, = \, \{ g_{\tau (1)} \ldots  g_{\tau (\ell)} \in G \mid \tau \mbox{ a permutation of $[1, \ell]$} \} \, \subset \, G \quad \und \quad \Pi_n (S) \, = \, \underset{|T|=n}{\bigcup_{T \t S}} \pi (T) \, \subset \, G \,,
\]
the {\it set of products} and {\it $n$-products} of $S$, and more generally, the {\it subsequence products} of $S$ is denoted by
\[
  \Pi (S) \, = \, \bigcup_{n \ge 1} \Pi_n (S) \, \subset \, G \,.
\]
It can easily be seen that $\pi(S)$ is contained in a $G'$-coset, where $G'$ is the commutator subgroup of $G$. Note that $|S|=0$ if and only if $S = 1_{\mathcal F (G)}$, and in that case we use the convention that $\pi (S) = \{ 1_G \}$. The sequence $S$ is called
\begin{itemize}
\item a {\it product-one sequence} if $1_G \in \pi (S)$,

\smallskip
\item {\it product-one free} if $1_G \notin \Pi (S)$, and

\smallskip
\item {\it square-free} if $\mathsf h (S) \le 1$.
\end{itemize}
If $S = g_1 \bdot \ldots \bdot g_{\ell} \in \mathcal B (G)$ is a product-one sequence with $1_G = g_1 \ldots g_{\ell}$, then $1_G = g_i \ldots g_{\ell}g_1 \ldots g_{i-1}$ for every $i \in [1, \ell]$.
Every map of groups $\theta : G \rightarrow H$ extends to a monoid homomorphism $\theta : \mathcal F (G) \rightarrow \mathcal F (H)$, where $\theta (S) = \theta (g_1)\bdot \ldots \bdot \theta (g_{\ell})$.
If $\theta$ is a group homomorphism, then $\theta (S)$ is a product-one sequence if and only if $\pi(S) \cap \ker(\theta) \neq \emptyset$.
We denote by
\[
  \mathcal B (G_0) \, = \, \big\{ S \in \mathcal F (G_0) \t 1_G \in \pi(S) \big\}
\]
the set of all product-one sequences over $G_0$, and clearly $\mathcal B (G_0) \subset \mathcal F (G_0)$ is a submonoid. We denote by $\mathcal A (G_0)$ the set of irreducible elements of $\mathcal B (G_0)$ which, in other words, is the set of minimal product-one sequences over $G_0$.
Moreover,
\[
  \mathsf D (G_0) \, = \, \sup \big\{ |S| \t S \in \mathcal A (G_0) \big\} \, \in \, \N \cup \{ \infty \}
\]
is the {\it large Davenport constant} of $G_0$, and
\[
  \mathsf d (G_0) \, = \, \sup \big\{ |S| \t S \in \mathcal F (G_0) \mbox{ is product-one free } \big\} \, \in \, \N_0 \cup \{ \infty \}
\]
is the {\it small Davenport constant} of $G_0$.

\smallskip
\noindent
{\bf Ordered sequences over groups.} These are an important tool used to study (unordered) sequences over non-abelian groups. Indeed, it is quite useful to have related notation for sequences in which the order of terms matters. Thus, for a subset $G_0 \subset G$, we denote by $\mathcal F^{*} (G_0) = \big( \mathcal F^{*} (G_0), \bdot \big)$ the free (non-abelian) monoid with basis $G_0$, whose elements will be called the {\it ordered sequences} over $G_0$.

Taking an ordered sequence in $\mathcal F^{*} (G_0)$ and considering all possible permutations of its terms gives rise to a natural equivalence class in $\mathcal F^{*} (G_0)$, yielding a natural map
\[
  [ \bdot ] \, : \, \mathcal F^{*} (G_0) \quad \to \quad \mathcal F (G_0)
\]
given by abelianizing the sequence product in $\mathcal F^{*} (G_0)$. For any sequence $S \in \mathcal F (G_0)$, we say that an ordered sequence $S^{*} \in \mathcal F^{*} (G_0)$ with $[S^{*}] = S$ is an {\it ordering} of the sequence $S \in \mathcal F (G_0)$.

All notation and conventions for sequences extend naturally to ordered sequences. We sometimes associate an (unordered) sequence $S$ with a fixed (ordered) sequence having the same terms, also denoted by $S$. While somewhat informal, this does not give rise to confusion, and will improve the readability of some of the arguments.

For an ordered sequence $S = g_1 \bdot \ldots \bdot g_{\ell} \in \mathcal F^{*} (G)$, we denote by $\pi^{*} : \mathcal F^{*} (G) \to G$ the unique homomorphism that maps an ordered sequence onto its product in $G$, so
\[
  \pi^{*} (S) \, = \, g_1 \ldots g_{\ell} \, \in \, G \,.
\]
If $G$ is a multiplicatively written abelian group, then for every sequence $S \in \mathcal F (G)$, we always use $\pi^{*} (S) \in G$ to be the unique product, and $\Pi (S) = \bigcup \big\{ \pi^{*} (T) \t  T \text{ divides } S \und |T| \ge 1 \big\} \subset G$.

\begin{definition} \label{2.1}~
We denote by
\begin{itemize}
\item $\mathsf s (G)$ the smallest integer $\ell$ such that every sequence of length at least $\ell$ has a product-one subsequence of length $\exp (G)$, and

\smallskip
\item $\mathsf E (G)$ the smallest integer $\ell$ such that every sequence of length at least $\ell$ has a product-one subsequence of length $|G|$.
\end{itemize}
\end{definition}

Note that $\max \{ \ord(g) \t g \in G \} \le \exp (G)$, and equality holds for nilpotent groups $G$. If $G$ is a dihedral group of order $2n$ with $n$ even, then the equality holds true, whence our definition for $\mathsf s (G)$ coincides with the one defined in the papers \cite{Cz-Do-Ge16,Ha-Zh19}. However, when $n$ is odd, the concepts are different. To see this, let $n \in \N_{\ge 3}$ be odd and let $G = \la \alpha, \tau \t \alpha^{n} = \tau^{2} = 1_G \text{ and } \tau\alpha = \alpha^{-1}\tau \}$ be a dihedral group of order $2n$. Then $\max \{ \ord(g) \t g \in G \} = n$ is odd, and it follows by the relation between $\alpha$ and $\tau$ that every sequence over $G \setminus \la \alpha \ra$ cannot have a product-one subsequence of any odd length, in particular of length $n$.

\smallskip
\begin{lemma} \label{2.2}~
Let $G$ be a finite group. Then $\mathsf s (G) \ge \mathsf d (G) + \exp (G)$ and $\mathsf E (G) \ge \mathsf d (G) + |G|$.
\end{lemma}

\begin{proof}
We need to find a sequence of length $\mathsf d (G) + \exp (G) - 1$ (or $\mathsf d (G) + |G| - 1$ respectively) which has no product-one subsequence of length $\exp (G)$ (or $|G|$ respectively). Take a product-one free sequence $S$ over $G$ of length $|S| = \mathsf d (G)$. Then $S \bdot 1^{[\exp(G)-1]}_G$ (or $S \bdot 1^{[|G|-1]}_G$ respectively) is the desired sequence.
\end{proof}

For the rest of this section, we list the following preliminary results.

\smallskip
\begin{lemma} \cite[Lemma 7]{Ga-Lu08a} \label{2.3}~
Let $G$ be a finite abelian group of order $n$, let $r \in \N_{\ge 2}$, and let $S \in \mathcal F (G)$ be a sequence of length $|S| = n + r -2$. If $S$ has no product-one subsequence of length $n$, then $|\Pi_{n-2} (S)| = |\Pi_r (S)| \ge r - 1$.
\end{lemma}

The following lemma can be found in \cite[Theorem 1]{Ga97} directly, and also in \cite[Theorem 5.1.16 and Proposition 5.1.14]{Ge09a} as a consequence of the characterization of long sequences having required property.

\smallskip
\begin{lemma} \label{2.4}~
Let $G$ be a cyclic group of order $n \ge 3$ and let $S \in \mathcal F (G)$ be a sequence of length $|S| = 2n - k \ge \frac{3n-1}{2}$ for $k \ge 2$. If $S$ has no product-one subsequence of length $n$, then there exist $g, h \in G$ with $\ord(gh^{-1}) = n$ such that $g^{[u]} \bdot h^{[v]} \t S$, where $u, v \ge n - 2k + 3$ and $u + v \ge 2n - 2k + 2$.
\end{lemma}

Finally, we conclude the following direct result from the previous lemma.

\smallskip
\begin{lemma} \label{2.5}~
Let $G$ be a cyclic group of order $n \ge 2$, and let $S \in \mathcal F (G)$ be a sequence of length $|S| = 3n-2$. Then the following statements are equivalent{\rm \,:}
\begin{enumerate}
\item[(a)] $S$ has no product-one subsequence of length $2n$.

\smallskip
\item[(b)] $S = g^{[2n-1]} \bdot h^{[n-1]}$, where $g, h \in G$ with $\ord (gh^{-1}) = n$.
\end{enumerate}
\end{lemma}

\begin{proof}
(b) $\Rightarrow$ (a) Since any sequence of the form $g^{[n-1]} \bdot h^{[n-1]}$, where $g, h \in G$ with $\ord(gh^{-1}) = n$, has no product-one subsequence of length $n$, the sequence $S = g^{[2n-1]} \bdot h^{[n-1]}$ has no product-one subsequence of length $2n$.

\smallskip
(a) $\Rightarrow$ (b) Let $S \in \mathcal F (G)$ be a sequence of length $|S| = 3n-2$. Suppose that $S$ has no product-one subsequence of length $2n$. Then it follows by $\mathsf s (G) = 2n-1$ that $S$ has a product-one subsequence $S_1$ of length $|S_1| = n$, whence $|S \bdot S^{[-1]}_1| = 2n-2$. Then Lemma \ref{2.4} ensures that there exist $g, h \in G$ with $\ord(gh^{-1}) = n$ such that $S = S_1 \bdot g^{[n-1]} \bdot h^{[n-1]}$.
If $n = 2$, then we clearly obtain that $S_1 = g^{[2]}$ or $S_1 = h^{[2]}$. Thus we suppose that $n \ge 3$, and assume to the contrary that there exists $x \in \supp(S_1)$ such that $x \neq g, h$. Then it follows again by Lemma \ref{2.4} that $x \bdot g^{[n-2]} \bdot h^{[n-1]}$ has a product-one subsequence of length $n$, say $x \bdot g^{[r_1]} \bdot h^{[r_2]}$, where $r_1, r_2 \in [1,n-3]$ such that $r_1 + r_2 = n-1$. Then $(S_1 \bdot x^{[-1]}) \bdot g^{[n-1 - r_1]} \bdot h^{[n-1 - r_2]}$ is a sequence of length $2n-2$. Since $S$ has no product-one subsequence of length $2n$, it follows again by Lemma \ref{2.4} that $S_1 \bdot x^{[-1]} = g^{[r_1]} \bdot h^{[r_2]}$, whence $g^{[n]} \bdot h^{[n]}$ is a product-one subsequence of $S$ of length $2n$, a contradiction. Thus every element in $\supp(S_1)$ is either $g$ or $h$. If $g, h \in \supp(S_1)$, then $g^{[n]} \bdot h^{[n]}$ is a product-one subsequence of $S$ of length $2n$, a contradiction. Therefore we obtain either $S_1 = g^{[n]}$ or $S_1 = h^{[n]}$.
\end{proof}

%%%%%%%%%%%%%%%%%%%%%%%%%%%%%%%%%%%%%%%%%%%%%%%%%%%%%%%%%%%%%%%%%%%%%%%%%%%%%%%%%%%%%%%%%%%%%%%%%%%%%%%%%%%%%%%%%%%%%%%%%%%%%%%%%%%%%%%%%%%%%%%%%%%%%%%%%%%%%%%%%%%%%%%%%%%
%%%%%%%%%%%%%%%										%%%%%%%%%%%%%%%										%%%%%%%%%%%%%%%										%%%%%%%%%%%%%%%
%%%%%%%%%%%%%%%%%%%%%%%%%%%%%%%%%%%%%%%%%%%%%%%%%%%%%%%%%%%%%%%%%%%%%%%%%%%%%%%%%%%%%%%%%%%%%%%%%%%%%%%%%%%%%%%%%%%%%%%%%%%%%%%%%%%%%%%%%%%%%%%%%%%%%%%%%%%%%%%%%%%%%%%%%%%

\bigskip
\section{On Dihedral groups} \label{3}
\bigskip

In this section, let $n \in \N_{\ge 3}$ and let $G$ be a dihedral group of order $2n$. Then we denote by $H$ the cyclic subgroup of $G$ of index $2$, and by $G_0 = G \setminus H$. Note that $\mathsf d (G) = n$ and $\exp (G) = \lcm(2,n)$.

We prove the technical lemmas by improving the argument from \cite{Ba07}, and then drive our main results.

\smallskip
\begin{lemma} \label{3.1}~
Let $n \in \N_{\ge 4}$ be even and let $S \in \mathcal F (G)$ be a sequence of length $|S| = 2n - 1$. If $S$ has no product-one subsequence of length $n$, then $|S_{G_0}| = 1$.
\end{lemma}

\begin{proof}
Let $\alpha, \tau \in G$ such that $G = \la \alpha, \tau \t \alpha^{n} = \tau^{2} = 1_G \text{ and } \tau\alpha = \alpha^{-1}\tau \ra$. Then $H = \la \alpha \ra$.
Let $S \in \mathcal F (G)$ be a sequence of length $|S| = 2n - 1$ such that $S$ has no product-one subsequence of length $n$. Note that $|S_{G_0}|\ge 1$. Assume to the contrary that $|S_{G_0}| \ge 2$.

Let $H_1 = \langle \alpha^{2} \rangle$, $H_2 = H \setminus H_1$, $G_1 = \big\{ \alpha^{2k}\tau \, \t \, k \in [0, \frac{n}{2}-1] \big\}$, and $G_2 = G_0 \setminus G_1$.
Suppose	that $S_{H_1} = a_1 \bdot \ldots \bdot a_r$, $S_{H_2} = b_1 \bdot \ldots \bdot b_s$, $S_{G_1} = c_1 \bdot \ldots \bdot c_t$, and $S_{G_2} = d_1 \bdot \ldots \bdot d_{\ell}$, where $r, s, t, \ell \in \N_0$. Then $r + s + t + \ell = 2n - 1$ is odd. We set
\[
  \begin{aligned}
    T_1 \,\, = & \,\,\, (a_1 a_2) \bdot \ldots \bdot (a_{2 \lfloor \frac{r}{2} \rfloor -1} a_{2 \lfloor \frac{r}{2} \rfloor}) \bdot (b_1 b_2) \bdot \ldots \bdot (b_{2 \lfloor \frac{s}{2} \rfloor - 1} b_{2 \lfloor \frac{s}{2} \rfloor}) \bdot \\
       	       & \,\,\, (c_1 c_2) \bdot \ldots \bdot (c_{2 \lfloor \frac{t}{2} \rfloor - 1} c_{2 \lfloor \frac{t}{2} \rfloor}) \bdot (d_1 d_2) \bdot \ldots \bdot (d_{2 \lfloor \frac{\ell}{2} \rfloor - 1} d_{2 \lfloor \frac{\ell}{2} \rfloor}) \,.
  \end{aligned}
\]
Then $T_1 \in \mathcal F (H_1)$ of length $|T_1| = \blfloor \frac{r}{2} \brfloor + \blfloor \frac{s}{2} \brfloor + \blfloor \frac{t}{2} \brfloor + \blfloor \frac{\ell}{2} \brfloor \ge n - 2 = 2 \big( \frac{n}{2} \big) - 2$.
Since $S$ has no product-one subsequence of length $n$, we have that $T_1$ has no product-one subsequence of length $\frac{n}{2}$. It follows by $\mathsf s (H_1) = 2 \big( \frac{n}{2} \big) - 1$ that $|T_1| = n - 2$ and three elements of $\{ r, s, t, \ell \}$ are odd. By Lemma \ref{2.4}, there exist $g, h \in H_1$ with $\ord (gh^{-1}) = \frac{n}{2}$ such that
\[
  T_1 \, = \, g^{[\frac{n}{2}-1]} \bdot h^{[\frac{n}{2}-1]} \,.
\]

\smallskip
\noindent
{\bf CASE 1.} \, $n = 4$.
\smallskip

If $\{ r, s, t, \ell \} \subset \N$, then $a_1 b_1, \, c_1 d_1 \in H_2$. Since $|H_2| = 2$, we obtain that either $a_1 b_1 = c_1 d_1$ or $a_1 b_1 = (c_1 d_1)^{-1} = d_1 c_1$.
It follows that either $c_1 \bdot a_1 \bdot b_1 \bdot d_1$ or $a_1 \bdot b_1 \bdot c_1 \bdot d_1$ is a product-one subsequence of $S$ of length $4$, a contradiction. Thus one element of $\{ r, s, t, \ell \}$ must be zero.

Suppose that $r = 0$, and the case when $s = 0$ follows by the same argument.
Then all $s, t, \ell$ must be odd with $s + t + \ell = 7$. Since $|H_1| = |H_2| = |G_1| = |G_2| = 2$, we obtain that there exist subsequences $W_1, W_2$ of $S$ such that $W_1 = x^{[2]}$ and $W_2 = y^{[2]}$ for some $x, y \in G$. If $x, y \in G_0$, then $W_1 \bdot W_2$ is a product-one subsequence of $S$ of length $4$, a contradiction. If $x \in H_2$ and $y \in G_0$ (or else $x \in G_0$ and $y \in H_2$ respectively), then $x y x y = 1_G$ (or $y x y x = 1_G$ respectively), which implies that $W_1 \bdot W_2$ is a product-one subsequence of $S$ of length $4$, a contradiction. If $x, y \in H_2$, then $x = y$ or $x = y^{-1}$, which implies that $W_1 \bdot W_2$ is a product-one subsequence of $S$ of length $4$, a contradiction.

Suppose that $t = 0$, and the case when $\ell = 0$ follows by the same argument.
Then all $r, s, \ell$ must be odd with $\ell \ge 3$ and $r + s + \ell = 7$. Since $|H_1| = |H_2| = |G_1| = |G_2| = 2$, we obtain that there exist subsequences $W_1, W_2$ of $S$ such that $W_1 = x^{[2]}$ and $W_2 = y^{[2]}$, where $x \in G$ and $y \in G_2$. If $x \in G_2$, then $W_1 \bdot W_2$ is a product-one subsequence of $S$ of length $4$, a contradiction. If $x \in H$, then $x y x y = 1_G$, which implies that $W_1 \bdot W_2$  is a product-one subsequence of $S$ of length $4$, a contradiction.

\smallskip
\noindent
{\bf CASE 2.} \, $n \ge 6$ and $t + \ell \ge 3$.
\smallskip

Then $t \ge 2$ or $\ell \ge 2$. If there exist distinct $t_1, t_2 \in [1,t]$ (or $\ell_1, \ell_2 \in [1,\ell]$ respectively) such that $c_{t_1} = c_{t_2}$ (or $d_{\ell_1} = d_{\ell_2}$ respectively), then after renumbering if necessary, we may assume that $c_1 = c_2$ (or $d_1 = d_2$ respectively). By symmetry, we may suppose that $c_1 c_2 = g$ (or $d_1 d_2 = g$ respectively). Then $h^{[\frac{n-2}{4}]} \bdot c_1 \bdot h^{[\frac{n-2}{4}]} \bdot c_2$ \big(or $h^{[\frac{n-2}{4}]} \bdot d_1 \bdot h^{[\frac{n-2}{4}]} \bdot d_2$ respectively\big) is a product-one sequence, and by splitting elements equal to $h$ to be subsequence of $S$ of length $2$, we infer that $S$ has a product-one subsequence of length $n$, a contradiction. Therefore
\[
  \mathsf h (S_{G_0}) \, = \, 1 \,, \quad \mbox{ and hence } \quad t + \ell \, \le \, n \,.
\]

We only prove the case when $t \ge 2$, and the case when $\ell \ge 2$ follows by the same argument. Now we assume that $t \ge 2$, and without loss of generality that $c_1 c_2 = g$.

Suppose that $\frac{n}{2}$ is even. Then $\frac{n}{2} - 1 \ge 3$, and $h^{[\frac{n-4}{4}]} \bdot c_1 \bdot h^{[\frac{n-4}{4}]} \bdot g \bdot c_2$ is a product-one sequence. It follows by splitting the elements equals $h$ and $g$ to be subsequences of $S$ length $2$ that there exists a product-one subsequence of $S$ of length $n$, a contradiction.

Suppose that $\frac{n}{2}$ is odd, and both $t$ and $\ell$ are odd. Then we set $T_2 = T_1 \bdot (c_1 c_2)^{[-1]} \bdot (c_1 c_t)$ which still has no product-one subsequence of length $\frac{n}{2}$. Comparing with $T_1$, we have $c_1 c_2 = c_1 c_t$, and hence $c_2 = c_t$, contradicting that $\mathsf h (S_{G_0}) = 1$.

Suppose that $\frac{n}{2}$ is odd, and both $r$ and $s$ are odd. Then we set $T_3 = T_1 \bdot (a_1 a_2)^{[-1]} \bdot (a_1 a_r)$ which still has no product-one subsequence of length $\frac{n}{2}$. Comparing with $T_1$, we have $a_1 a_2 = a_1 a_r$, and we obtain by proceeding the same argument that $S_{H_1} = a^{[r]}_1$ and $S_{H_2} = b^{[s]}_1$.
Since $r + s = 2n - 1 - (t + \ell) \ge n - 1 \ge 5$, there exist $x, y \in H$ such that $x^{[2]} \bdot y^{[2]} \t S_H$. If $n \ge 10$, then $h^{[\frac{n-6}{4}]} \bdot c_1 \bdot h^{[\frac{n-6}{4}]} \bdot g \bdot x \bdot c_2\bdot x$ is a product-one sequence, and by splitting the elements equal to $h$ and $g$ to be subsequences of $S$ of length $2$, we infer that $S$ has a product-one subsequence of length $n$, a contradiction. If $n = 6$, then $x^{2} \neq y^{2}$, and hence either $c_1 \bdot x^{[2]} \bdot y \bdot c_2 \bdot y$ or $c_1 \bdot y^{[2]} \bdot x \bdot c_2 \bdot x$ is a product-one subsequence of $S$ of length $6$, a contradiction.

\smallskip
\noindent
{\bf CASE 3.} \, $n \ge 6$ and $t + \ell = 2$.
\smallskip

Then $|S_H| = 2n - 3$. Since $t$ or $\ell$ must be odd, we have $t=\ell=1$.
Suppose that $S_{G_1} = \alpha^{i}\tau$ and $S_{G_2} = \alpha^{j}\tau$, where $i, j \in [0,n-1]$.
Since $S_H$ has no product-one subsequence of length $n$, it follows by lemma \ref{2.4} that there exist $r_1, r_2 \in [0,n-1]$ such that $r_1$ is odd, $r_2$ is even, and  $\gcd(r_1 - r_2, n) = 1$ such that $(\alpha^{r_1})^{[u_1]} \bdot (\alpha^{r_2})^{[u_2]} \t S_H $ for some $u_1, u_2 \ge n-3 \ge \frac{n}{2}$ with $u_1 + u_2 \ge 2n - 4 \ge n + 2$.
Then there exist $v, w \in [0,n-1]$ such that
\[
  i \, \equiv \, v (r_1 - r_2) \, \pmod{n} \quad \und \quad j \, \equiv \, w (r_1 - r_2) \, \pmod{n} \,,
\]
and we set $x$ to be $|w-v|$ if $|w-v| \le \frac{n}{2}$, and $n - |w-v|$ if $|w-v| > \frac{n}{2}$.
Then $x \in [1,\frac{n}{2}]$ is odd, and $V = (\alpha^{r_1})^{[x]} \bdot (\alpha^{r_2})^{[x]} \bdot \alpha^{i}\tau \bdot \alpha^{j}\tau$ is a product-one subsequence of $S$ having even length. If $x \le \frac{n}{2} - 1$, then by adding even terms of either $\alpha^{r_1}$ or $\alpha^{r_2}$ to $V$, we obtain a product-one subsequence of $S$ of length $n$, a contradiction. Therefore $x = \frac{n}{2}$ is odd and $j \equiv \frac{n}{2} + i \pmod{n}$.

Suppose that $\gcd(r_1,n) > 1$. Since $r_1$ is odd, we have $\gcd(r_1,n) \t \frac{n}{2}$, and hence there exists $x_0 \in [1, \frac{n}{2}-1]$ such that $r_1 x_0 \equiv \frac{n}{2} \pmod{n}$. Since $\frac{n}{2}$ is odd, we obtain that $x_0$ is odd. Hence $n - 2 x_0 - 4 \ge 0$ and
\[
  (\alpha^{r_1})^{[x_0+\frac{n-2x_0-4}{4}]} \bdot (\alpha^i\tau) \bdot (\alpha^{r_1})^{[\frac{n-2x_0-4}{4}]} \bdot (\alpha^j\tau) \bdot (\alpha^{r_2})^{[\frac{n}{2}]}
\]
is a product-one subsequence of $S$ of length $n$, a contradiction.

Suppose that $\gcd(r_1,n) = 1$. If $r_2 = 0$, then since $u_1, u_2 \ge \frac{n}{2}$, it follows that $(\alpha^{r_2})^{[\frac{n}{2}-2]} \bdot (\alpha^{r_1})^{[\frac{n}{2}]} \bdot \alpha^{i}\tau \bdot \alpha^{j}\tau$ is a product-one subsequence of $S$ of length $n$, a contradiction. Thus we assume that $r_2 \neq 0$, and then there exists an odd $x_0 \in [1,n] \setminus \big\{ \frac{n}{2} \big\}$ such that $r_1 x_0 \equiv \frac{n}{2} +r_2 \pmod{n}$. Thus $|\frac{n}{2} - x_0| \ge 2$. If $x_0 \in [1, \frac{n}{2}-1]$, then $n -2 x_0 - 4 \ge 0$ and
\[
  (\alpha^{r_1})^{[x_0+\frac{n-2x_0-4}{4}]} \bdot (\alpha^i\tau) \bdot (\alpha^{r_1})^{[\frac{n-2x_0-4}{4}]} \bdot (\alpha^{r_2})^{[1+\frac{n-2}{4}]} \bdot (\alpha^j\tau) \bdot (\alpha^{r_2})^{[\frac{n-2}{4}]}
\]
is a product-one subsequence of $S$ of length $n$, a contradiction. If $x_0 \in [\frac{n}{2}+1, n]$, then $2 x_0 - n - 4 \ge 0$ and
\[
  (\alpha^{r_1})^{[n-x_0+\frac{2x_0-n-4}{4}]}\bdot(\alpha^i\tau)\bdot (\alpha^{r_1})^{[\frac{2x_0-n-4}{4}]}\bdot (\alpha^{r_2})^{[1+\frac{n-2}{4}]}\bdot(\alpha^j\tau)\bdot (\alpha^{r_2})^{[\frac{n-2}{4}]}
\]
is a product-one subsequence of $S$ of length $n$, a contradiction.
\end{proof}

\smallskip
\begin{proof}[Proof of Theorem \ref{1.1}.1]
If $n \ge 3$ is odd, then $\mathsf s (G) = \mathsf E (G)$ and the assertion follows from \cite{Ba07}. Suppose that $n \ge 4$ is even, and we assert that $\mathsf s(G) = 2n$. By Lemma \ref{2.2}, it suffices to show that every sequence of length $2n$ has a product-one subsequence of length $n$. Let $S \in \mathcal F (G)$ be a sequence of length $2n$, and assume to the contrary that $S$ has no product-one subsequence of length $n$.

Let $T \t S$ be a subsequence of length $2n-1$. Then $T$ has no product-one subsequence of length $n$. By Lemma \ref{3.1}, we have $|T_{G_0}| = 1$, and hence $|S \bdot T_{G_0}^{[-1]}| = 2n - 1$. Again by Lemma \ref{3.1}, we have $\bigl\vert (S\bdot T_{G_0}^{[-1]})_{G_0} \bigl\vert \, = 1$, which implies that $|S_{G_0}| = 2$.
Let $W \t S$ be a subsequence of length $2n-1$ with $S_{G_0} \t W$. Then Lemma \ref{3.1} ensures that $W$, and hence $S$, has a product-one subsequence of length $n$, a contradiction.
\end{proof}

\smallskip
\begin{lemma} \label{3.2}~
Let $n \in \N_{\ge 3}$ and let $S \in \mathcal F (G)$ be a sequence of length $|S| = 3n - 1$. If $S$ has no product-one subsequence of length $2n$, then either $|S_{G_0}| = 1$, or that $n = 3$ and $S = 1^{[5]}_G \bdot g_1 \bdot g_2 \bdot g_3$, where $\{ g_1, g_2, g_3 \} = G \setminus H$.
\end{lemma}

\begin{proof}
Let $\alpha, \tau \in G$ such that $G = \la \alpha, \tau \t \alpha^{n} = \tau^{2} = 1_G \text{ and } \tau\alpha = \alpha^{-1}\tau \ra$.
Then $H = \la \alpha \ra$, and if $n = 3$, then $G \setminus H = \{ \tau, \alpha\tau, \alpha^{2}\tau \}$.
Let $S \in \mathcal F (G)$ be a sequence of length $|S| = 3n-1$ such that $S$ has no product-one subsequence of length $2n$. Note that $|S_{G_0}| \ge 1$.
Assume to the contrary that $|S_{G_0}| \ge 2$ and $S \neq 1^{[5]}_G \bdot \tau \bdot \alpha\tau \bdot \alpha^{2}\tau$ when $n = 3$.

\smallskip
1. We first prove the case when $n \ge 4$ is even. Let $g_1, g_2 \in G_0$ be such that $g_1 \bdot g_2 \t S_{G_0}$. Then $|S \bdot (g_1 \bdot g_2)^{[-1]}| = 3n-3 \ge 2n$, and by Theorem \ref{1.1}.1, there exists a product-one subsequence $T_1 \t S \bdot (g_1 \bdot g_2)^{[-1]}$ of length $n$. Then $|S \bdot T^{[-1]}_1| = 2n - 1$ and $| (S \bdot T^{[-1]}_1)_{G_0}| \ge |g_1 \bdot g_2| = 2$, whence Lemma \ref{3.1} ensures that there exists a product-one subsequence $T_2 \t S \bdot T^{[-1]}_1$ of length $n$. Therefore $T_1 \bdot T_2$ is a product-one subsequence of $S$ of length $2n$, a contradiction.

\smallskip
2. Now we prove the case when $n \ge 3$ is odd. By renumbering if necessary, we can assume that
\[
  S_{G_0} \, = \, \big( a^{[2]}_1 \bdot \ldots \bdot a^{[2]}_r \big) \bdot (c_1 \bdot \ldots \bdot c_u) \,,
\]
where $c_i \neq c_j$ for all distinct $i, j \in [1,u]$. Then $u \le n$. Similarly we set
\[
  S_H \, = \, \big( b^{[2]}_1 \bdot \ldots \bdot b^{[2]}_t \big) \bdot (e_1\bdot e_1^{-1})\bdot\ldots\bdot (e_s\bdot e_s^{-1})\bdot (d_1 \bdot \ldots \bdot d_v) \,,
\]
where $d_i \notin \{ d_j, d_j^{-1} \}$ for all distinct $i, j \in [1,v]$, and $\ord(b_k) \neq 2$ for all $k \in [1,v]$. Then $v \le \frac{n+1}{2}$.

\smallskip
\noindent
{\bf CASE 1.} \, $r \ge 1$.
\smallskip

Let $W_1$ be the maximal product-one subsequence of $c_1 \bdot \ldots \bdot c_u\bdot d_1 \bdot \ldots \bdot d_v$ of even length. Then $|W_1|\le 2n$ and $W = c_1 \bdot \ldots \bdot c_u \bdot d_1 \bdot \ldots \bdot d_v \bdot W^{[-1]}_1$ has no product-one subsequence of even length. If $|W| \le n-1$, then $2t + 2r + 2s + |W_1| \ge 2n$, and hence there exist $t_1 \in [0, t]$, $r_1 \in [0, r]$, and $s_1 \in[0, s]$ with $2t_1 + 2r_1 + 2s_1 + |W| = 2n$ such that
\[
  W_1 \bdot (e_1 \bdot e_1^{-1}) \bdot \ldots \bdot (e_{s_1} \bdot e_{s_1}^{-1}) \bdot (b_1 \bdot \ldots\bdot b_{t_1}) \bdot a_1 \bdot (b_1 \bdot \ldots \bdot b_{t_1}) \bdot a_1 \bdot a^{[2]}_2 \bdot \ldots \bdot a^{[2]}_{r_1}
\]
is a product-one subsequence of $S$ of length $2n$, a contradiction. Since $n$ is odd, we have $|W| \ge n+1$.

Suppose that $|W_H| \ge 2$. Let $\Omega = \big\{ x \in H \, \t \, x \in \Pi_2 (W_H) \, \text{ or } \, x^{-1} \in \Pi_2 (W_H) \big\}$. Since $\Pi_2 (W_H) \subset H \setminus \{ 1_G \}$, we have $|\Omega| \ge |W_H|$, and hence $|\Pi_2 (W_{G_0})| + |\Omega| \ge \big( |W_{G_0}|- 1 \big) + |W_H| \ge n = |H|$. It follows by $\Pi_2 (W_{G_0}) \cup \Omega \subset H \setminus \{ 1_G \}$ that there exist ordered sequences $T_1 \t W_{G_0}$ and $T_2 \t W_H$ with $|T_1| = |T_2| = 2$ such that $\pi^{*}(T_1) = \pi^{*}(T_2) \in \Pi_2 (W_{G_0}) \cap \Omega$, which implies that $T_1 \bdot T_2$ is a product-one subsequence of $W$ of length $4$, a contradiction to the maximality of $W_1$.

Suppose that $|W_H| = 1$. Then $|W_{G_0}| = n$. If $n\ge 5$, then $W_{G_0}$ has a product-one subsequence of length $4$, a contradiction to the maximality of $W_1$. If $n = 3$, then it is easy to verify that there are $g_1, g_2 \in \supp(W_{G_0})$ such that $T = a^{[2]}_1 \bdot g_1 \bdot g_2$ is a product-one sequence. Note that $4 = 2r + 2t + 2s + |W_1|$. If $W_1$ is non-trivial, then since $r \ge 1$, we must have $|W_1| = 2$, whence $W_1 \bdot T$ is a product-one subsequence of $S$ of length $6$, a contradiction. Thus $W_1$ is a trivial sequence, and then $r + t + s = 2$. Hence we infer that $S$ has a product-one subsequence of length $6$, a contradiction.

\smallskip
\noindent
{\bf CASE 2.} \, $r = 0$.
\smallskip

Then $2 \le |S_{G_0}| = u \le n$, and we proceed by the following assertion.

\smallskip
\begin{itemize}
\item[{\bf A.}] {\it For every non-trivial product-one subsequence $W$ of $S_{G_0}$, we have $u - |W| \ge \frac{n+1}{2}$.}
\end{itemize}

\begin{proof}[Proof of {\bf A}]
Assume to the contrary that $S_{G_0} = c_1 \bdot \ldots \bdot c_u$ has a non-trivial product-one subsequence $W$ such that $u - |W| \le \frac{n-1}{2}$. Since $|W|$ is even and
\[
  \bigl\vert \, \big( b^{[2]}_1 \bdot \ldots \bdot b^{[2]}_t \big) \bdot (e_1 \bdot e_1^{-1}) \bdot \ldots \bdot (e_s \bdot e_s^{-1}) \bdot W \, \bigl\vert \,\, \ge \,\, (3n - 1) - \frac{n+1}{2} - \frac{n-1}{2} \, = \, 2n - 1 \,,
\]
it follows that $\big( b^{[2]}_1 \bdot \ldots \bdot b^{[2]}_t \big) \bdot (e_1 \bdot e_1^{-1}) \bdot \ldots \bdot (e_s \bdot e_s^{-1}) \bdot W$ has a product-one subsequence of length $2n$, a contradiction.
\end{proof}

\noindent
{\bf SUBCASE 2.1.} \, $u = n$.
\smallskip

Since all $c_i$ are distinct, we may assume by renumbering if necessary that $c_i = \alpha^{i}\tau$ for every $i \in [1,n-1]$ and $c_u = c_n = \tau$.
If $n = 3$, then $S_{G_0} = \tau \bdot \alpha\tau \bdot \alpha^{2}\tau$ and $|S_H| = 5$. If $\alpha^{i} \in \Pi_4 (S_H)$ for some $i \in [1,2]$, then $S$ has a product-one subsequence of length $6$, a contradiction. Thus $\Pi_4 (S_H) = \{ 1_G \}$, which implies that $S_H = 1^{[5]}_G$ and $S = \tau \bdot \alpha\tau \bdot \alpha^{2}\tau \bdot 1^{[5]}_G$, a contradiction to our assumption.

Suppose that $n \ge  5$. If $\frac{n-1}{2}$ is even, then $\alpha\tau \bdot \ldots \bdot \alpha^{n-1}\tau$ is a product-one subsequence of $S_{G_0}$, a contradiction to {\bf A}. If $\frac{n-1}{2}$ is odd, then $n \ge 7$ and $\tau \bdot \alpha^{2}\tau \bdot \alpha^{3}\tau \bdot \ldots \bdot \alpha^{n-1}\tau$ is a product-one subsequence of $S_{G_0}$, again a contradiction to {\bf A}.

\smallskip
\noindent
{\bf SUBCASE 2.2.} \, $u = 2$.
\smallskip

Let $S_{G_0} = \alpha^{i}\tau \bdot \alpha^{j}\tau$ for distinct $i, j \in [0,n-1]$. If $n = 3$, then $|S_H| = 6$ and $\pi(S_{G_0}) = \{ \alpha^{i-j}, \, \alpha^{j-i} \} = \{ \alpha, \alpha^{2} \}$. If $\alpha^{k} \in \Pi_{4} (S_H)$ for some $k \in [1,2]$, then $S$ has a product-one subsequence of length $6$, a contradiction. Thus $\Pi_{4} (S_H) = \{ 1_G \}$, and hence $S_H = 1^{[6]}_G$ is a product-one subsequence of length $6$, a contradiction.

Suppose that $n \ge 5$. Since $|S_H| = 3n - 3$, it follows by $\mathsf s (H) = 2n - 1$ that there exists a product-one subsequence $T_1 \t S_H$ of length $n$, whence $|S_H \bdot T^{[-1]}_1| = 2n-3$. Since $S$ has no product-one subsequence of length $2n$, it follows by Lemma \ref{2.4} that there exist $r_1, r_2 \in [0,n-1]$ with $\gcd(r_1 - r_2, n) = 1$ such that
\[
  (\alpha^{r_1})^{[\gamma_1]} \bdot (\alpha^{r_2})^{[\gamma_2]} \, \bigl\vert \,\, S_H \bdot T^{[-1]}_1 \,, \,\, \mbox{ where } \,\, \gamma_1, \gamma_2 \, \ge \, n - 3 \, \ge \, \frac{n-1}{2} \und \gamma_1 + \gamma_2 \, \ge \, 2n - 4 \, \ge \, n + 1 \,.
\]
Since $\gcd(r_1-r_2, n) = 1$, there exists $x \in [0,n-1]$ such that $(r_1 - r_2) x \equiv j - i + r_1 \pmod{n}$. If $x \le \frac{n-1}{2}$, then $V = (\alpha^{r_1})^{[x-1]} \bdot \alpha^{i}\tau \bdot (\alpha^{r_2})^{[x]} \bdot \alpha^{j}\tau$ is a product-one subsequence of $S$. Since $\gamma_1 + \gamma_2 \ge n + 1$, by adding even terms of either $\alpha^{r_1}$ or $\alpha^{r_2}$ to $V$, we obtain that $S \bdot T_1^{[-1]}$ has a product-one subsequence of length $n$, a contradiction.
If $x \ge \frac{n+3}{2}$, then $n - x + 1 \le \frac{n-1}{2}$ and $V = (\alpha^{r_1})^{[n-x+1]} \bdot \alpha^{j}\tau \bdot (\alpha^{r_2})^{[n-x]} \bdot \alpha^{i}\tau$ is a product-one subsequence of $S$. Since $\gamma_1 + \gamma_2 \ge n + 1$, by adding even terms of either $\alpha^{r_1}$ or $\alpha^{r_2}$ to $V$, we obtain that $S \bdot T_1^{[-1]}$  has a product-one subsequence of length $n$, a contradiction. Thus $x = \frac{n+1}{2}$, and again by $\gcd(r_1-r_2, n) = 1$, there exists $y \in [0,n-1]$ such that $(r_1 - r_2) y \equiv i - j + r_1 \pmod{n}$. A similar argument shows that $y = \frac{n+1}{2} = x$, and hence $i - j + r_1 \equiv j - i + r_1 \pmod{n}$. Since $n$ is odd, it follows that $i = j$, a contradiction.

\smallskip
\noindent
{\bf SUBCASE 2.3.} \, $u \in [3, n-1]$.
\smallskip

Then $|S_H| = 3n - u - 1 \ge \mathsf s (H) = 2n - 1$ ensures that $S_H$ has a product-one subsequence $T_1$ of length $n$. Then $S_H \bdot T^{[-1]}_1$ is a sequence over $H$ of length $2n - u - 1$ which has no product-one subsequence of length $n$, and thus Lemma \ref{2.3} ensures that $\bigl\vert \Pi_{n-2} \big( S_H \bdot T^{[-1]}_1 \big) \bigl\vert \, \ge n-u$. Since all $c_i$ are distinct, we have that all $c_1c_2, c_1c_3, \ldots, c_1c_u$ are distinct, and hence $\bigl\vert \Pi_2 (c_1 \bdot \ldots \bdot c_u) \bigl\vert \, \ge  u - 1$.

\smallskip
\noindent
{\bf SUBCASE 2.3.1} \, $\bigl\vert \Pi_2 (c_1 \bdot \ldots \bdot c_u) \bigl\vert \, \ge u$.
\smallskip

If $\Pi_{n-2} \big( S_H \bdot T^{[-1]}_1 \big) \cap \Pi_2 (c_1 \bdot \ldots \bdot c_u) \neq \emptyset$, then there exists a product-one subsequence $T_2 \t S \bdot T^{[-1]}_1$ of length $n$, a contradiction.
Thus $\Pi_{n-2} \big( S_H \bdot T^{[-1]}_1 \big) \cap \Pi_2 (c_1 \bdot \ldots \bdot c_u) = \emptyset$, implying that $\bigl\vert \Pi_2 (c_1 \bdot \ldots \bdot c_u) \bigl\vert \, = u$ and $\bigl\vert \Pi_{n-2} \big( S_H \bdot T^{[-1]}_1 \big) \bigl\vert \, = n - u$.
Note that if $h \in \Pi_2 (c_1 \bdot \ldots \bdot c_u)$, then $h^{-1} \in \Pi_2 (c_1 \bdot \ldots \bdot c_u)$, for otherwise, $S \bdot T^{[-1]}_1$ must have a product-one subsequence of length $n$, again a contradiction.
Since $n$ is odd and $\Pi_2 (c_1 \bdot \ldots \bdot c_u) \subset H \setminus \{ 1_G \}$, we must have that $u$ is even. Since all $c_1c_2, c_1c_3, \ldots, c_1c_u$ are distinct, after renumbering if necessary, we may assume that $c_1 c_{2k} = (c_1 c_{2k+1})^{-1}$ for all $k \in [1, \frac{u}{2} - 1]$, and suppose that $c_i = \alpha^{r_i}\tau$ for all $i \in [1, u]$. Then $2r_1 \equiv r_2 + r_3 \equiv \ldots \equiv r_{u-2} + r_{u-1} \pmod{n}$, and hence for every $k \in \N_0$ such that $(2 + 4k) + 3 \le u - 1$, we have
\[
  c_{2+4k} \, c_{(2+4k)+2} \, c_{(2+4k)+1} \, c_{(2+4k)+3} \,\, = \,\, \alpha^{r_{2+4k}+r_{(2+4k)+1}-r_{(2+4k)+2}-r_{(2+4k)+3}} \,\, = \,\, 1_G \,.
\]

Suppose that $u \ge 6$ with $4 \t u-2$. Then $W = c_2 \bdot \ldots \bdot c_{u-1}$ is a product-one sequence of length $u-2$. Since $n \ge u \ge 6$, we have $u - |W| = 2 \le \frac{n-1}{2}$, a contradiction to {\bf A}.

Suppose that $u \ge 6$ with $4 \nmid u-2$.  Then $4 \t u-4$ and $W = c_2 \bdot \ldots \bdot c_{u-3}$ is a product-one sequence of length $u-4$. Since $u \ge 8$, we obtain $n \ge 9$, and hence $u - |W| = 4 \le \frac{n-1}{2}$, again a contradiction to {\bf A}.

Suppose that $u = 4$. Then $c_3c_4\in \Pi_{2} (c_1 \bdot \ldots \bdot c_4) = \{ c_1c_2, \, c_1c_3, \, c_1c_4, \, c_4c_1 \}$. Since all $c_i$ are distinct, we have $c_3c_4 \neq c_1c_4$. If $c_3c_4 = c_1c_2$ or $c_3c_4 = c_1c_3 = c_2c_1$, then $c_1 \bdot \ldots \bdot c_4$ is a product-one sequence, a contradiction to {\bf A}, whence $c_3c_4 = c_4c_1$. Similarly we can prove that $c_2c_4 = c_4c_1$, which implies that $c_2c_4 = c_3c_4$, and thus $c_2 = c_3$, a contradiction.

\smallskip
\noindent
{\bf SUBCASE 2.3.2} \, $\bigl\vert \Pi_2 (c_1 \bdot \ldots \bdot c_u) \bigl\vert \, = u - 1$.
\smallskip

For every $i \in [1,u]$, we set $c_i = \alpha^{r_i}\tau$. Since $\bigl\vert \Pi_2 (c_1 \bdot \ldots \bdot c_u) \bigl\vert \, = u-1$, it follows that
\[
  \Pi_2 (c_1 \bdot \ldots \bdot c_u) \, = \, \{ c_1c_2, \ldots, c_1c_u \} \, = \, \big\{ \alpha^{r_1 - r_2}, \alpha^{r_1 - r_3}, \ldots, \alpha^{r_1 - r_u} \big\} \, = \, \big\{ \alpha^{r_2 - r_1}, \alpha^{r_2 - r_3}, \ldots, \alpha^{r_2 - r_u} \big\} \,.
\]
By multiplying all the elements of each set, we have the modulo equation
\[
  (u-1) r_1 - (r_2 + r_3 + \ldots + r_u) \, \equiv \, (u-1)r_2 - (r_1 + r_3 + \ldots + r_u) \, \pmod{n} \,,
\]
which implies that $ur_1 \equiv ur_2 \pmod{n}$. Similarly, $ur_i \equiv ur_j \pmod{n}$ for all distinct $i, j \in [1,u]$. Thus all elements $\alpha^{r_i - r_j}$ are in the subgroup of $H$ having order $\gcd(u,n)$. But $u = \bigl\vert \{ 1_G, \alpha^{r_1 - r_2}, \ldots, \alpha^{r_1 - r_u} \} \bigl\vert \, \le \gcd(u,n) \le u$, whence $\gcd(u,n) = u$ and $u \t n$. Since $n$ is odd, $u$ is also odd, and $\Pi_2 (c_1 \bdot \ldots \bdot c_u) = \big\{ \alpha^{\frac{n}{u}}, \ldots, \alpha^{(u-1)\frac{n}{u}} \big\}$. By renumbering if necessary, we can assume that
\[
  c_1 \, = \, \alpha^{r_1}\tau, \quad c_2 \, = \, \alpha^{r_1 + \frac{n}{u}}\tau, \quad \ldots, \quad c_u \, = \, \alpha^{r_1 + (u-1)\frac{n}{u}}\tau \,.
\]

Suppose that $u \ge 5$ and $4 \t u-1$. Then $c_1 \bdot c_2 \bdot c_4 \bdot c_3$ is a product-one sequence, and thus $c_1 \bdot \ldots \bdot c_{u-1}$ is a product-one subsequence of $S_{G_0}$ of length $ u-1$, a contradiction to {\bf A}.

Suppose that $u \ge 5$ and $4 \nmid u-1$. Then $u \ge 7$, and $c_1 \bdot c_3 \bdot c_5 \bdot c_4 \bdot c_7 \bdot c_6$ is a product-one sequence. Thus $c_1 \bdot c_3 \bdot c_4 \bdot \ldots \bdot c_{u-1}$ is a product-one subsequence of $S_{G_0}$ of length $u-1$, again a contradiction to {\bf A}.

Suppose that $u = 3$. Since $u \t n$ and $u\le n-1$, we obtain $n \ge 9$, and since $|S_H| = 3n - 4 \ge \mathsf s(H) = 2n - 1$, we obtain that $S_H$ has a product-one subsequence $T_1$ of length $n$. Then $|S_H \bdot T^{[-1]}_1| = 2n - 4$ and $S_H \bdot T_1^{[-1]}$ has no product-one subsequence of length $n$. It follows by Lemma \ref{2.4} that there exist $r_1, r_2 \in [0,n-1]$ with $\gcd(r_1 - r_2, n) = 1$ such that
\[
  (\alpha^{r_1})^{[\gamma_1]} \bdot (\alpha^{r_2})^{[\gamma_2]} \, \bigl\vert \, S_H \bdot T^{[-1]}_1 \quad \mbox{ for some } \,\, \gamma_1, \gamma_2 \, \ge \, n - 5 \, \ge \, \frac{n-1}{2} \und \gamma_1 + \gamma_2 \, \ge \, 2n - 6 \, \ge \, n + 3 \,.
\]
Since $\gcd(r_1-r_2, n) = 1$, there exists $x \in [0,n-1]$ such that $(r_1-r_2) x \equiv \frac{n}{3} + r_1 \pmod{n}$.
If $x \le \frac{n-1}{2}$, then $V = (\alpha^{r_1})^{[x-1]} \bdot c_1 \bdot (\alpha^{r_2})^{[x]} \bdot c_2$ is a product-one subsequence of $S$. Since $\gamma_1 + \gamma_2 \ge n+3$, by adding even terms of either $\alpha^{r_1}$ or $\alpha^{r_2}$ to $V$, we obtain that $S \bdot T_1^{[-1]}$ has a product-one subsequence of length $n$, a contradiction.
If $x \ge \frac{n+3}{2}$, then $n - x + 1 \le \frac{n-1}{2}$ and $V = (\alpha^{r_1})^{[n-x+1]} \bdot c_2 \bdot (\alpha^{r_2})^{[n-x]} \bdot c_1$ is a product-one subsequence of $S$. Since $\gamma_1 + \gamma_2 \ge n+3$, by adding even terms of either $\alpha^{r_1}$ or $\alpha^{r_2}$ to $V$, we obtain that $S \bdot T_1^{[-1]}$ has a product-one subsequence of length $n$, a contradiction. Thus $x = \frac{n+1}{2}$, and again by $\gcd(r_1-r_2, n) = 1$, there exists $y \in [0,n-1]$ such that $(r_1-r_2) y \equiv \frac{2n}{3} + r_1 \pmod{n}$. A similar argument shows that $y = \frac{n+1}{2} = x$, and hence $\frac{n}{3} + r_1 \equiv \frac{2n}{3} + r_1 \pmod{n}$, a contradiction.
\end{proof}

\smallskip
\begin{proof}[Proof of Theorem \ref{1.2}]
1. (b) $\Rightarrow$ (a) Suppose that $\alpha, \tau \in G$ are generators, $H = \la \alpha \ra$, and $S = (\alpha^{r_1})^{[n-1]} \bdot (\alpha^{r_2})^{[n-1]} \bdot \alpha^{r_3}\tau$, where $r_1, r_2, r_3 \in [0,n-1]$ with $\gcd(r_1 - r_2, n) = 1$. It follows that $S_H = g^{[n-1]} \bdot h^{[n-1]}$ with $g, h \in H$ and $\ord(gh^{-1}) = n$. Then $S_H$ has no product-one subsequence of length $n$, and thus the assertion follows.

\smallskip
(a) $\Rightarrow$ (b) Let $n \ge 4$ be even. Then $\mathsf s (G) = 2n$ by Theorem \ref{1.1}.1. Thus the assertion follows by Lemmas \ref{3.1} and \ref{2.4}.

\medskip
2. (b) $\Rightarrow$ (a) Clearly, $S = 1^{[5]}_G \bdot g_1 \bdot g_2 \bdot g_3$, where $n = 3$ and $\{ g_1, g_2, g_3 \} = G \setminus H$, has no product-one subsequence of length $6$, and the remains follow from Lemma \ref{2.5}.

\smallskip
(a)$\Rightarrow$(b) Let $n \ge 3$. Note that $\mathsf E(G) = 3n$. Then the assertion follows by Lemmas \ref{3.2} and \ref{2.5}.
\end{proof}

%%%%%%%%%%%%%%%%%%%%%%%%%%%%%%%%%%%%%%%%%%%%%%%%%%%%%%%%%%%%%%%%%%%%%%%%%%%%%%%%%%%%%%%%%%%%%%%%%%%%%%%%%%%%%%%%%%%%%%%%%%%%%%%%%%%%%%%%%%%%%%%%%%%%%%%%%%%%%%%%%%%%%%%%%%%
%%%%%%%%%%%%%%%										%%%%%%%%%%%%%%%										%%%%%%%%%%%%%%%										%%%%%%%%%%%%%%%
%%%%%%%%%%%%%%%%%%%%%%%%%%%%%%%%%%%%%%%%%%%%%%%%%%%%%%%%%%%%%%%%%%%%%%%%%%%%%%%%%%%%%%%%%%%%%%%%%%%%%%%%%%%%%%%%%%%%%%%%%%%%%%%%%%%%%%%%%%%%%%%%%%%%%%%%%%%%%%%%%%%%%%%%%%%

\bigskip
\section{On Dicyclic groups} \label{4}
\bigskip

In this section, let $n \in \N_{\ge 2}$ and let $G$ be a dicyclic group of order $4n$. Then we denote by $H$ a cyclic subgroup of $G$ of index $2$, and by $G_0 = G \setminus H$. Note that $H$ is unique if $n \ge 3$, and we observe that $\mathsf d (G) = 2n$ and $\exp (G) = \lcm(4,2n)$. We denote by 
\[
  \psi \, \colon \, G\, \rightarrow \, \overline{G} \, = \, G / \mathsf Z (G) \,,
\]
where $\mathsf Z (G)$ is the center of $G$, the natural epimorphism. If $n \ge 3$, then $\overline{G}$ is a dihedral group of order $2n$, and thus we also denote by $\overline{H}$ the cyclic subgroup of $\overline{G}$ of index $2$, and by $\overline{G}_0 = \overline{G} \setminus \overline{H}$.

Likewise, we prove the technical lemmas by using results from the dihedral case and by improving the argument from \cite{Ba07}, and then derive our main results.

\smallskip
\begin{lemma} \label{4.1}~
Let $n \in \N_{\ge 2}$ be even and let $S \in \mathcal F (G)$ be a sequence of length $|S| = 4n - 1$.
\begin{enumerate}
\item If $n \ge 4$ and $S$ has no product-one subsequence of length $2n$, then $|S_{G_0}| = 1$.

\smallskip
\item If $n = 2$ and $S$ has no product-one subsequence of length $4$, then for any $\alpha, \tau \in G$ with $G = \la \alpha, \tau \t \alpha^{4} = 1_G, \tau^{2} = \alpha^{2}, \text{ and } \tau\alpha = \alpha^{-1}\tau \ra$, we have that $S$ has one of the following forms{\rm \,:}
\[\hspace{35pt}
  S \, \in \, \big\{ (\alpha^{r_1})^{[3]} \bdot (\alpha^{r_2})^{[3]} \bdot \alpha^{r_3}\tau \,, \,\, (\alpha^{r_1})^{[3]} \bdot (\alpha^{r_3}\tau)^{[3]} \bdot \alpha^{r_2} \,, \,\, (\alpha^{r_1})^{[3]} \bdot (\alpha^{r_3}\tau)^{[3]} \bdot \alpha^{r_4}\tau \big\} \,,
\]
where $r_1, r_2, r_3, r_4 \in [0,3]$ such that $r_1$ is even, $r_2$ is odd, and $r_3 \not\equiv r_4 \pmod{2}$. In particular, $|\supp(S)| = 3$.

\end{enumerate}
\end{lemma}

%If $n = 2$ and $S$ has no product-one subsequence of length $4$, then $S$ has one of the following forms{\rm \,:}
%      \[\hspace{25pt}
%        S \, \in \, \big\{ a^{[3]} \bdot b^{[3]} \bdot c \,, \quad a^{[3]} \bdot b^{[3]} \bdot d \,, \quad a^{[3]} \bdot c^{[3]} \bdot b \,, \quad a^{[3]} \bdot d^{[3]} \bdot b \,, \quad a^{[3]} \bdot c^{[3]} \bdot d \,, \quad a^{[3]} \bdot d^{[3]} \bdot c \big\} \,,
%      \]
%      where $G/\mathsf Z (G) = \big\{ H_1, H_2, G_1, G_2 \big\}$ with $H_1 = \mathsf Z (G)$ and $H = H_1 \cupdot H_2$, $a \in H_1$, $b \in H_2$, $c \in G_1$, and $d \in G_2$. In particular, $|\supp(S)| = 3$.

\begin{proof}
Let $\alpha, \tau \in G$ such that $G = \la \alpha, \tau \t \alpha^{2n} = 1_G, \tau^{2} = \alpha^{n}, \text{ and } \tau\alpha = \alpha^{-1}\tau \ra$ and let $H = \la \alpha \ra$. Let $S \in \mathcal F (G)$ be a sequence of length $|S| = 4n-1$ such that $S$ has no product-one subsequence of length $2n$. Note that $|S_{G_0}| \ge 1$.

\smallskip
1. Let $n \ge 4$. Assume to the contrary that $|S_{G_0}| \ge 2$. Then $\psi(S)$ is a sequence of length $4n-1$ over $\overline{G}$ and $|\psi(S)_{\overline{G}_0}| \ge 2$. Let $g_1, g_2 \in G_0$ be such that $g_1 \bdot g_2 \t S_{G_0}$. Since $|\psi(S \bdot (g_1 \bdot g_2)^{[-1]})| = 4n - 3 \ge 2n$, it follows by Theorem \ref{1.1}.1 that there exists $T_1 \t S\bdot (g_1\bdot g_2)^{[-1]}$ of length $n$ such that $\psi(T_1)$ is a product-one sequence.
Since $|\psi(S \bdot (g_1 \bdot g_2 \bdot T_1)^{[-1]})| = 3n - 3 \ge 2n$, it follows again by Theorem \ref{1.1}.1 that there exists $T_2 \t S \bdot (g_1 \bdot g_2 \bdot T_1)^{[-1]}$ of length $n$ such that $\psi(T_2)$ is a product-one sequence.
Therefore $|\psi(S \bdot (T_1 \bdot T_2)^{[-1]})| = 2n - 1$ and $|\psi(S \bdot (T_1 \bdot T_2)^{[-1]})_{\overline{G}_0}| \ge |\psi(g_1\bdot g_2)| = 2$. We deduce by Lemma \ref{3.1} that there exists $T_3 \t S \bdot (T_1 \bdot T_2)^{[-1]}$ of length $n$ such that $\psi(T_3)$ is a product-one sequence. Since $\pi(T_i) \cap \{ 1_G, \alpha^{n} \} \neq \emptyset$ for all $i \in [1,3]$, we obtain that there exist distinct $i, j \in [1,3]$ such that $T_i \bdot T_j$ is a product-one subsequence of $S$ of length $2n$, a contradiction.

\smallskip
2. Let $n = 2$. If $|S_{G_0}| = 1$, then $|S_H| = 6$, and hence Lemma \ref{2.4} ensures that $S$ has the desired structure. We assume that $|S_{G_0}| \ge 2$. 
Let $H_1 = \{ 1_G, \alpha^{2} \}$, $H_2 = \{ \alpha, \alpha^{3} \}$, $G_1 = \{ \tau, \alpha^{2}\tau \}$, and $G_2 = \{ \alpha\tau, \alpha^{3}\tau \}$. Suppose that $S_{H_1} = a_1 \bdot \ldots \bdot a_r$, $S_{H_2} = b_1 \bdot \ldots \bdot b_s$, $S_{G_1} = c_1 \bdot \ldots \bdot c_t$, and $S_{G_2} = d_1 \bdot \ldots \bdot d_{\ell}$, where $r, s, t, \ell \in \N_0$. Then $r + s + t + \ell =7$. We set
\[
  \begin{aligned}
     T_1 \,\, = & \,\,\, (a_1 a_2) \bdot \ldots \bdot (a_{2 \lfloor \frac{r}{2} \rfloor -1} a_{2 \lfloor \frac{r}{2} \rfloor}) \bdot (b_1 b_2) \bdot \ldots \bdot (b_{2 \lfloor \frac{s}{2} \rfloor - 1} b_{2 \lfloor \frac{s}{2} \rfloor}) \bdot \\
                & \,\,\, (c_1 c_2) \bdot \ldots \bdot (c_{2 \lfloor \frac{t}{2} \rfloor - 1} c_{2 \lfloor \frac{t}{2} \rfloor}) \bdot (d_1 d_2) \bdot \ldots \bdot (d_{2 \lfloor \frac{\ell}{2} \rfloor - 1} d_{2 \lfloor \frac{\ell}{2} \rfloor}) \,.
  \end{aligned}
\]
Then $T_1 \in \mathcal F (H_1)$ of length $|T_1| = \blfloor \frac{r}{2} \brfloor + \blfloor \frac{s}{2} \brfloor + \blfloor \frac{t}{2} \brfloor + \blfloor \frac{\ell}{2} \brfloor \ge 2$. Since $S$ has no product-one subsequence of length $4$, we obtain that $T_1$ has no product-one subsequence of length $2$. Since $|H_1| = 2$, we obtain that $|T_1| = 2$, and hence three elements of $\{ r, s, t, \ell \}$ are odd.

If $\{ r, s, t, \ell \}\subset \N$, then $a_1 b_1, \, c_1 d_1 \in H_2$. Since $|H_2| = 2$, we obtain that either $a_1 b_1 = c_1 d_1$ or $a_1 b_1 = (c_1 d_1)^{-1}$. It follows that either $c_1 \bdot b_1 \bdot a_1 \bdot d_1$ or $a_1 \bdot b_1 \bdot c_1 \bdot d_1$ is a product-one subsequence of $S$ of length $4$, a contradiction. Thus one element of $\{ r, s, t, \ell \}$ must be zero, and we distinguish four cases.

Suppose that $r = 0$. Then all $s, t, \ell$ must be odd with $s + t + \ell = 7$. Since $|H_1| = |H_2| = |G_1| = |G_2| = 2$, we obtain that there exist subsequences $W_1, W_2$ of $S$ such that $W_1 = x^{[2]}$ and $W_2 = y^{[2]}$ for some $x, y \in G$. If $x \in G_0$ or $y \in G_0$, then $W_1 \bdot W_2$ is a product-one subsequence of $S$ of length $4$, a contradiction. If $x, y \in H_2$, then $x = y$ or $x = y^{-1}$, which implies that $W_1 \bdot W_2$ is a product-one subsequence of $S$ of length $4$, a contradiction.

Suppose that $s = 0$. Then all $r, t, \ell$ must be odd with $r + t + \ell = 7$. Since $|H_1| = |H_2| = |G_1| = |G_2| = 2$, we obtain that there exist subsequences $W_1, W_2$ of $S$ such that $W_1 = x^{[2]}$ and $W_2 = y^{[2]}$ for some $x, y \in G$. If $x, y \in H_1$ or $x, y \in G_0$, then $W_1 \bdot W_2$ is a product-one subsequence of $S$ of length $4$, a contradiction. Thus we may assume that $x \in H_1$ and $y \in G_0$, whence $r = 3$ and $t + \ell = 4$. Thus either $S = \big( a^{[2]}_1 \bdot a_2 \big) \bdot \big( c^{[2]}_1 \bdot c_2 \big) \bdot d_1$ or $S = \big( a^{[2]}_1 \bdot a_2 \big) \bdot c_1 \bdot \big( d^{[2]}_1 \bdot d_2 \big)$. In the former case, if $a_1 \neq a_2$ (or $c_1 \neq c_2$ respectively), then $a_1 \bdot a_2 \bdot c^{[2]}_1$ (or $a^{[2]}_1 \bdot c_1 \bdot c_2$ respectively) is a product-one subsequence of $S$ of length $4$, a contradiction, whence $S = a^{[3]}_1 \bdot c^{[3]}_1 \bdot d_1$ has the desired structure. In the latter case, we similarly obtain $S = a^{[3]}_1 \bdot c_1 \bdot d^{[3]}_1$, which has the desired structure.

Suppose that $t = 0$. Then all $r, s, \ell$ must be odd with $\ell \ge 3$ and $r + s + \ell = 7$. Since $|H_1| = |H_2| = |G_1| = |G_2| = 2$, we obtain that there exist subsequences $W_1, W_2$ of $S$ such that $W_1 = x^{[2]}$ and $W_2 = y^{[2]}$, where $x \in G$ and $y \in G_2$. If $x \in G_2$, then $W_1 \bdot W_2$ is a product-one subsequence of $S$ of length $4$, a contradiction. Thus we must have $\ell = 3$ and $r + s = 4$. If $s = 3$, then $x\in H_2$, and hence $W_1 \bdot W_2$ is a product-one subsequence of $S$ of length $4$, a contradiction. Hence we must have $r = 3$ and $s = 1$, and thus after renumbering if necessary, we have $S = \big( a^{[2]}_1 \bdot a_2 \big) \bdot b_1 \bdot \big( d^{[2]}_1 \bdot d_2 \big)$. If $a_1 \neq a_2$ (or $d_1 \neq d_2$ respectively), then $a_1 \bdot a_2 \bdot d_1 \bdot d_1$ (or $a_1 \bdot a_1 \bdot d_1 \bdot d_2$ respectively) is a product-one subsequence of $S$ of length $4$, a contradiction. Therefore $S = a^{[3]}_1 \bdot b_1 \bdot d^{[3]}_1$ has the desired structure.

Suppose that $\ell = 0$. A similar argument as used in the case $t = 0$ shows that $S = a_1^{[3]} \bdot b_1 \bdot c_1^{[3]}$, which has the desired structure.
\end{proof}

\smallskip
\begin{proof}[Proof of Theorem \ref{1.1}.2]
If $n \ge 3$ is odd, then $\mathsf s (G) = \mathsf E (G)$ and the assertion follows from \cite{Ba07}. Suppose that $n \ge 2$ is even, and we assert that $\mathsf s (G) = 4n$. By Lemma \ref{2.2}, it suffices to show that every sequence of length $4n$ has a product-one subsequence of length $2n$. Let $S \in \mathcal F (G)$ be a sequence of length $4n$, and assume to the contrary that $S$ has no product-one subsequence of length $2n$.
Let $T \t S$ be a subsequence of length $4n-1$. Then $T$ has no product-one subsequence of length $2n$.

Suppose that $n \ge 4$. Then Lemma \ref{4.1}.1 ensures that $|T_{G_0}| = 1$, and in addition that $\bigl\vert ( S \bdot T^{[-1]}_{G_0})_{G_0} \bigl\vert \, = 1$, implying $|S_{G_0}| = 2$. Let $W \t S$ be a subsequence of length $4n-1$ with $S_{G_0} \t W$. Again by Lemma \ref{4.1}.1, we obtain that $W$, and hence $S$, has a product-one subsequence of length $2n$, a contradiction.

Suppose that $n = 2$. Then $T$ must have the desired structure in Lemma \ref{4.2}.2. We fix $\alpha, \tau \in G$ such that $G = \la \alpha, \tau \t \alpha^{4} = 1_G, \tau^{2} = \alpha^{2}, \text{ and } \tau \alpha = \alpha^{-1}\tau \ra$. If $|T_{G_0}| = 1$, then $S = x \bdot (\alpha^{r_1})^{[3]} \bdot (\alpha^{r_2})^{[3]} \bdot \alpha^{r_3}\tau$ for some $x \in G$ and $r_1, r_2, r_3 \in [0,3]$ with $\gcd(r_1 - r_2, 4) = 1$. Since $S \bdot (\alpha^{r_1})^{[-1]}$ is a sequence of length $7$ having no product-one subsequence of length $4$, Lemma \ref{4.1}.2 ensures that $x = \alpha^{r_1}$, whence $(\alpha^{r_1})^{[4]}$ is a product-one subsequence of $S$ of length $4$, a contradiction. If $|T_{G_0}| \ge 2$, then the same argument shows that $(\alpha^{r_1})^{[4]}$ is a product-one subsequence of $S$ of length $4$, a contradiction.
\end{proof}

\smallskip
\begin{lemma} \label{4.2}~
Let $n \in \N_{\ge 2}$ and let $S \in \mathcal F (G)$ be a sequence of length $|S| = 6n - 1$.
\begin{enumerate}
\item If $n \ge 3$ and $S$ has no product-one subsequence of length $4n$, then $|S_{G_0}| = 1$.

\smallskip
\item If $n = 2$ and $S$ has no product-one subsequence of length $8$, then for any $\alpha, \tau \in G$ with $G = \la \alpha, \tau \t \alpha^{4} = 1_G, \tau^{2} = \alpha^{2}, \text{ and } \tau\alpha = \alpha^{-1}\tau \ra$, we have that $S$ has one of the following forms{\rm \,:}
\[\hspace{35pt}
  S \, \in \,  \Big\{ \big( (\alpha^{r_1})^{[3]} \bdot (\alpha^{r_2})^{[3]} \bdot \alpha^{r_3}\tau \big) \bdot S_0, \, \big( (\alpha^{r_1})^{[3]} \bdot (\alpha^{r_3}\tau)^{[3]} \bdot \alpha^{r_2} \big) \bdot S_1, \, \big( (\alpha^{r_1})^{[3]} \bdot (\alpha^{r_3}\tau)^{[3]}  \bdot \alpha^{r_4}\tau \big) \bdot S_1 \Big\} \,,
\]
where $S_0 \in \big\{ (\alpha^{r_1})^{[4]}, \, (\alpha^{r_2})^{[4]} \big\}$, $S_1 \in \big\{ (\alpha^{r_1})^{[4]}, \, (\alpha^{r_3}\tau)^{[4]} \big\}$, and $r_1, r_2, r_3, r_4 \in [0,3]$ such that $r_1$ is even, $r_2$ is odd, and $r_3 \not\equiv r_4 \pmod{2}$.
\end{enumerate}
\end{lemma}

%If $n = 2$ and $S$ has no product-one subsequence of length $8$, then $S$ has one of the following forms{\rm \,:}
%      \[\hspace{25pt}
%        \begin{aligned}
%          S \, & \in \, \big\{ a^{[7]} \bdot b^{[3]} \bdot c \,, \quad a^{[7]} \bdot b^{[3]} \bdot d \,, \quad a^{[7]} \bdot c^{[3]} \bdot b \,, \quad a^{[7]} \bdot c^{[3]} \bdot d \,, \quad a^{[7]} \bdot d^{[3]} \bdot b \,, \quad a^{[7]} \bdot d^{[3]} \bdot c \,, \\
%               & \hspace{25pt} a^{[3]} \bdot b^{[7]} \bdot c \,, \quad a^{[3]} \bdot b^{[7]} \bdot d \,, \quad a^{[3]} \bdot c^{[7]} \bdot b \,, \quad a^{[3]} \bdot c^{[7]} \bdot d \,, \quad a^{[3]} \bdot d^{[7]} \bdot b \,, \quad a^{[3]} \bdot d^{[7]} \bdot c \big\} \,,
%        \end{aligned}
%      \]
%      where $G/\mathsf Z (G) = \big\{ H_1, H_2, G_1, G_2 \big\}$ with $H_1 = \mathsf Z (G)$ and $H = H_1 \cupdot H_2$, $a \in H_1$, $b \in H_2$, $c \in G_1$, and $d \in G_2$.

\begin{proof}
Let $\alpha, \tau \in G$ such that $G = \la \alpha, \tau \t \alpha^{2n} = 1_G, \tau^{2} = \alpha^{n}, \text{ and } \tau\alpha = \alpha^{-1}\tau \ra$ and let $H = \la \alpha \ra$. Let $S \in \mathcal F (G)$ be a sequence of length $6n-1$ such that $S$ has no product-one subsequence of length $4n$. Note that $|S_{G_0}| \ge 1$.

\smallskip
1. Let $n \ge 3$. Assume to the contrary that $|S_{G_0}| \ge 2$. We first suppose that $n \ge 4$ is even. Let $g_1, g_2 \in G_0$ be such that $g_1 \bdot g_2 \t S_{G_0}$. Then $|S \bdot (g_1 \bdot g_2)^{[-1]}| = 6n - 3 \ge 4n$, and by Theorem \ref{1.1}.2, there exists a product-one subsequence $T_1 \t S \bdot (g_1 \bdot g_2)^{[-1]}$ of length $2n$. Then $|S \bdot T^{[-1]}_1| = 4n - 1$ and $|(S \bdot T^{[-1]}_1)_{G_0}| \ge |g_1 \bdot g_2| \ge 2$. Thus Lemma \ref{4.1}.1 ensures that $S \bdot T_1^{[-1]}$ has a product-one subsequence $T_2$ of length $2n$, whence $T_1 \bdot T_2$ is a product-one subsequnce of $S$ of length $4n$, a contradiction.

Suppose now that $n \ge 3$ is odd. We distinguish two cases.

\smallskip
\noindent
{\bf CASE 1.} \, $|S_{G_0}| \le n$.
\smallskip

Then $|S_H| \ge 5n-1$, and by mapping under $\psi$, we obtain a sequence $\psi(S_H) \in \mathcal F (\overline{H})$ of length at least $5n-1$. Since $5n-1 > 2n - 1 = \mathsf s (\overline{H})$, we obtain that there exists $T_1 \t S_H$ of length $n$ such that $\psi(T_1)$ is a product-one sequence over $\overline{H}$, whence $\pi^{*} (T_1) \in \{ 1_G, \alpha^{n} \}$. Continuing this process, we obtain that there exist subsequences $T_2, T_3, T_4$ of $S_H$ such that $|T_2| = |T_3| = |T_4| = n$ and $\pi^{*}(T_i) \in \{ 1_G, \alpha^{n} \}$ for all $i \in [2,4]$. Since $S$ has no product-one subsequence of length $4n$, we may assume that $\pi^{*}(T_1) \neq \pi^*(T_2)=\pi^{*}(T_3) = \pi^*(T_4)$. Consider the sequence $\psi \big( S \bdot (T_1 \bdot T_2 \bdot T_3)^{[-1]} \big) \in \mathcal F (\overline{G})$ of length $3n-1$.

Suppose that $n \ge 5$. Since $|S_{G_0}| \ge 2$, it follows by Lemma \ref{3.2} that there exists an ordered subsequence $T_5 \t S \bdot (T_1 \bdot T_2 \bdot T_3)^{[-1]}$ of length $2n$ such that $\pi^{*} (T_5) \in \{ 1_G, \alpha^{n} \}$. It follows that either $T_1 \bdot T_3 \bdot T_5$ or $T_2 \bdot T_3 \bdot T_5$ is a product-one subsequence of $S$ of length $4n$, a contradiction.

Suppose that $n = 3$. Since $|S_{G_0}| \ge 2$, Lemma \ref{3.2} ensures that
\[
  S \, = \, T_1 \bdot T_2 \bdot T_3 \bdot \big( \alpha^{r_1} \bdot \alpha^{r_2} \bdot \alpha^{r_3} \bdot \alpha^{r_4} \bdot \alpha^{r_5} \big) \bdot  \Big( \alpha^{r_6}\tau \bdot \alpha^{r_7}\tau \bdot \alpha^{r_8}\tau \Big) \,,
\]
where $r_1, \ldots, r_6 \in \{ 0, 3 \}$, $r_7 \in \{ 1, 4 \}$, and $r_8 \in \{ 2, 5 \}$. Since there exist distinct $i, j \in [6,8]$ such that $r_i \equiv r_j \pmod{2}$, we obtain $\{ \alpha, \alpha^{5} \} = \big\{ \alpha^{r_i}\tau\alpha^{r_j}\tau, \, \alpha^{r_j}\tau\alpha^{r_i}\tau \big\} \subset \Pi_2 (\alpha^{r_6}\tau \bdot \alpha^{r_7}\tau \bdot \alpha^{r_8}\tau )$, and assert that $S \bdot (T_2 \bdot T_3)^{[-1]}$ has a product-one subsequenc of length $6$, leading a contradiction to the fact that $S$ has no product-one subsequence of length $12$.

If $\pi^{*}(T_1) \in \Pi_3 (\alpha^{r_1} \bdot \ldots \bdot \alpha^{r_5})$, then $T_1 \bdot \alpha^{r_1} \bdot \ldots \bdot \alpha^{r_5}$ has a product-one subsequence of length $6$.
Suppose that $\bigl\vert \Pi_3 (\alpha^{r_1} \bdot \ldots \bdot \alpha^{r_5}) \bigl\vert \, = 1$ and $\pi^{*}(T_1) \notin \Pi_3 (\alpha^{r_1} \bdot \ldots \bdot \alpha^{r_5})$. Then $r_1 = \ldots = r_5$ and $\pi^{*}(T_1) \neq \alpha^{r_1}$, whence $\pi^{*}(T_1 \bdot \alpha^{r_1}) = \alpha^{3}$.
If $\{ 1_G, \alpha^{3} \} \subset \Pi_2 (T_1 \bdot \alpha^{r_1})$, then together with $\alpha^{r_2}\bdot\ldots\bdot \alpha^{r_5}$, we obtain a product-one subsequence of length $6$.  If $\{\alpha, \alpha^2\}\subset \Pi_2(T_1\bdot \alpha^{r_1})$ or $\{\alpha^4, \alpha^5\}\subset \Pi_2(T_1\bdot \alpha^{r_1})$, then $(T_1\bdot \alpha^{r_1}) \bdot \alpha^{r_6}\tau\bdot \alpha^{r_7}\tau \bdot \alpha^{r_8}\tau$ has a product-one subsequence of length $4$. Together with $\alpha^{r_2} \bdot \alpha^{r_3}$, we obtain a product-one subsequence of length $6$.

\smallskip
\noindent
{\bf CASE 2.} \, $|S_{G_0}| \ge n+1$.
\smallskip

By renumbering if necessary, we can assume that
\[
  S_{G_0} \, = \, (a_1\bdot a_1^{-1}) \bdot \ldots \bdot (a_r\bdot a_r^{-1}) \bdot \big( W_1\bdot \ldots W_x \big) \bdot  (c_1 \bdot \ldots \bdot c_u) \,,
\]
where $W_i$ is a product-one sequence of length $4$ for all $i \in [1,x]$, and $c_1 \bdot \ldots \bdot c_u$  has no product-one subsequence of length $2$ and $4$. Then $\mathsf h (c_1 \bdot \ldots \bdot c_u) \le 3$, and thus $u-2 \le |\supp(c_1 \bdot \ldots \bdot c_u)| \le n$, whence $u \le n+2$. Similarly we set
\[
  S_H \, = \, (e_1 \bdot e_1^{-1}) \bdot \ldots \bdot (e_s \bdot e^{-1}_s) \bdot \big( b^{[2]}_1 \bdot \ldots \bdot b^{[2]}_t \big) \bdot  (d_1 \bdot \ldots \bdot d_v) \,,
\]
where $d_i \notin \{ d_j, d_j^{-1} \}$ for all distinct $i, j \in [1,v]$, and $\ord(b_k) \neq 2$ for all $k \in [1,t]$. Then $d_1 \bdot \ldots \bdot d_v$ is square-free, and hence $v \le n+1$. Let $R = c_1 \bdot \ldots \bdot c_u \bdot d_1 \bdot \ldots \bdot d_v$.

\smallskip
\noindent
{\bf SUBCASE 2.1.} \, $r + s + t = 0$.
\smallskip

If $\mathsf h (R_{G_0}) = 1$, then $u \le n$, and hence $4x \ge (6n - 1) - u - v \ge 4n - 2$. It follows that $x \ge n$, whence $W_1 \bdot \ldots W_n$ is a product-one subsequence of $S$ of length $4n$, a contradiction. Thus we must have $\mathsf h (R_{G_0}) \ge 2$, and by renumbering if necessary, we may assume that $c_1 = c_2$.
If $v \ge \frac{n+1}{2} + 1$, then by mapping under $\psi$, we infer that there exist $i, j \in [1,v]$ such that $\pi^{*}(d_i d_j) = \alpha^{n}$, whence $c^{[2]}_1 \bdot d_i \bdot d_j$ is a product-one sequence of length $4$. It follows by $4x \ge 4n - 4$ that $W_1 \bdot \ldots \bdot W_{n-1} \bdot c^{[2]}_1 \bdot d_i \bdot d_j$ is a product-one subsequence of $S$ of length $4n$, a contradiction. Thus we must have $v \le \frac{n+1}{2}$.

Suppose that $n \ge 5$. Then $u + v \le \frac{3n+5}{2} \le 2n$, and since $u + v$ is odd, it follows that $u + v \le 2n - 1$. Hence $4x \ge 4n$, and therefore $W_1 \bdot \ldots \bdot W_n$ is a product-one subsequence of $S$ of length $4n$, a contradiction.

Suppose that $n = 3$. Then $u \le 5$ and $v \le 2$. If $u + v \le 5$, then $4x \ge 12$, whence $S$ has a product-one subsequence of length $12$, a contradiction. Thus we must have that $u = 5$ and $v = 2$, and then $4x = (6n - 1) - u - v = 10$, a contradiction.

\smallskip
\noindent
{\bf SUBCASE 2.2.} \, $r + s + t \ge 1$.
\smallskip

Suppose that $u \le n$. If $2r + 4x + 2s + 2t \ge 4n$, then $S$ must have a product-one subsequence of length $4n$. We assume that $2r + 4x + 2s + 2t \le 4n - 2$. Then $2n + 1 \ge u + v \ge (6n - 1) - (4n - 2) = 2n + 1$, and hence $u = n$ and $v = n+1$. Then we may assume by renumbering if necessary that $d_1 = 1_G$ and $d_2 = \alpha^{n}$. If $n \ge 5$, then there exist distinct $i, j \in [3, v]$, say $i = 3$ and $j = 4$, such that $d_3 d_4 \notin \{ 1_G, \alpha^{n} \}$. Since $|R_{G_0} \bdot (d_3 d_4)| = n + 1$, it follows by mapping under $\psi$ that there exists an ordered subsequence $T$ of $R_{G_0} \bdot d_3 \bdot d_4$ having even length such that $\pi^{*}(T) \in \{ 1_G, \alpha^{n} \}$. Since $R_{G_0}$ has no product-one subsequence of length $2$, we obtain that either $T$ or $T \bdot d_1 \bdot d_2$ is a product-one subsequence of $S$ having even length at least $4$. Since $2r + 4x + 2s + 2t = 4n - 2$ and $r + s + t \ge 1$, it follows that $S$ has a product-one subsequence of length $4n$, a contradiction.
If $n = 3$, then we obtain that $\{ \alpha, \alpha^{5} \} \subset \Pi_2 (R_{G_0})$ as argued in the case $n =3$ of {\bf CASE 1}. By renumbering if necessary, we may assume that $d_3 \in \{ \alpha, \alpha^{5} \}$ and $d_4 \in \{ \alpha^{2}, \alpha^{4} \}$. It follows that $\{ \alpha, \alpha^{5} \} \subset \Pi_2 (R_H)$, whence $R$ has a product-one subseqeunce $T$ of length $4$ with $|T_{G_0}| = 2$. Since $2r + 2x + 2s + 2t = 10$ and $r + s + t \ge 1$, it follows that $S$ has a product-one subseqeunce of legnth $12$, a contradiction.

Suppose that $u \ge n + 1$. Then $\mathsf h (R_{G_0}) \ge 2$, and by renumbering if necessary, we may assume that $c_1 = c_2$. If $t = 0$ and $v \le 1$, then $2r + 2s + 4x \ge 5n - 3 \ge 4n - 1$, and hence $S$ must have a product-one subsequence of length $4n$, a contradiction. Thus either $t \ge 1$ or $v \ge 2$. By mapping under $\psi$, \cite[Theorem 1.3]{Br-Ri18a} ensures that there exists an ordered sequence $T$ of $c_3 \bdot \ldots \bdot c_u \bdot d_1 \bdot d_2$ (or $c_3 \bdot \ldots \bdot c_u \bdot b^{[2]}_1$ if $v \le 1$) having even length such that $\pi^{*}(T) \in \{ 1_G, \alpha^{n} \}$. Since $R_{G_0}$ has no product-one subsequence of length $2$, we obtain that either $T$ or $T \bdot c^{[2]}_1$ is a product-one subsequence of $S$ having even length at least $4$. Since $2r + 4x + 2s + 2t \ge 4n - 4$ (or $2r + 4x + 2s + 2(t-1) \ge 5n - 6 \ge 4n - 4$ if $v \le 1$) and $r + s + t \ge 1$, it follows that $S$ has a product-one subsequence of length $4n$, a contradiction.

\smallskip
2. Let $n = 2$. If $|S_{G_0}| = 1$, then $|S_H| = 10$, and hence Lemma \ref{2.5} ensures that $S$ has the desired structure. We assume that $|S_{G_0}| \ge 2$. 
Let $g_1, g_2 \in G_0$ be such that $g_1 \bdot g_2 \t S_{G_0}$. Then $|S \bdot (g_1 \bdot g_2)^{[-1]}| = 9$, and by Theorem \ref{1.1}.2, there exists a product-one subsequence $T_1 \t S \bdot (g_1 \bdot g_2)^{[-1]}$ of length $4$. Then $|S \bdot T^{[-1]}_1| = 7$ and $|(S \bdot T^{[-1]}_1)_{G_0}| \ge |g_1 \bdot g_2| \ge 2$. Then Lemma \ref{4.1}.2 ensures that
\[
  S \bdot T^{[-1]}_1 \, = \, (\alpha^{r_1})^{[3]} \bdot (\alpha^{r_2}\tau)^{[3]} \bdot x \,,
\]
where $r_1, r_2 \in [0,3]$ such that $r_1$ is even and $x \in \{ \alpha, \alpha^{3}, \alpha^{r_3}\tau \}$ with $r_3 \not\equiv r_2 \pmod{2}$.
If $\alpha^{r_1} \in \supp(T_1)$, then $(\alpha^{r_1})^{[4]} \t S$, and we obtain by Lemma \ref{4.1}.2 that $S \bdot \big( (\alpha^{r_1})^{[4]} \big)^{[-1]} = x \bdot (\alpha^{r'})^{[3]} \bdot (\alpha^{r_2}\tau)^{[3]}$, where $r' \in [0,3]$ is even. Thus $T_1 = \alpha^{r_1} \bdot (\alpha^{r'})^{[3]}$ is a product-one sequence, which implies that $r' = r_1$. Therefore $S = x \bdot (\alpha^{r_1})^{[7]} \bdot (\alpha^{r_2}\tau)^{[3]}$ has the desired structure.
Suppose that $\alpha^{r_1} \notin \supp(T_1)$. Let $g \in \supp(T_1)$ be any element. Then $S \bdot (\alpha^{r_2}\tau \bdot x \bdot g)^{[-1]}$ has length $8$, and by Theorem \ref{1.1}.2, it has a product-one subsequence $T_2$ of length $4$. Then $S \bdot T^{[-1]}_2$ has no product-one subsequence of length $4$, and Lemma \ref{4.1}.2 ensures that $|\supp(S \bdot T^{[-1]}_2)| = 3$. Since $T_2 \neq (\alpha^{r_1})^{[4]}$, we have $\alpha^{r_1} \bdot \alpha^{r_2}\tau \bdot x \bdot g \t S \bdot T^{[-1]}_2$, whence $g \in \{ x, \alpha^{r_2}\tau \}$. Thus $T_1 \in \big\{ x^{[4]}, \, (\alpha^{r_2}\tau)^{[4]}, \, (x \bdot \alpha^{r_2}\tau)^{[2]} \big\}$.
If $T_1 = x^{[4]}$ (or $T_1 = (x \bdot \alpha^{r_2}\tau)^{[2]}$ respectively), then $(\alpha^{r_1})^{[2]} \bdot (x \bdot \alpha^{r_2}\tau)^{[2]} \bdot x^{[2]}$ (or $(\alpha^{r_1})^{[2]} \bdot (x \bdot \alpha^{r_2}\tau)^{[2]} \bdot (\alpha^{r_2}\tau)^{[2]}$ respectively) is a product-one subsequence of $S$ of length $8$, a contradiction.
Therefore $T_1 = (\alpha^{r_2}\tau)^{[4]}$, and $S$ has the desired structure.
\end{proof}

\smallskip
\begin{proof}[Proof of Theorem \ref{1.3}]
1. (b) $\Rightarrow$ (a) (1) Suppose that $\alpha, \tau \in G$ are generators, $H = \la \alpha \ra$, and $S = (\alpha^{r_1})^{[2n-1]} \bdot (\alpha^{r_2})^{[2n-1]} \bdot \alpha^{r_3}\tau$, where $r_1, r_2, r_3 \in [0, 2n-1]$ with $\gcd(r_1 - r_2, 2n) = 1$. It follows that $S_H = g^{[2n-1]} \bdot h^{[2n-1]}$ with $g, h \in H$ and $\ord(gh^{-1}) = 2n$. Then $S_H$ has no product-one subsequence of length $2n$, and thus the assertion follows.

\smallskip
(2) By (1), it suffices to verify that a sequence $S$ with $|S_{G_0}| \ge 2$ given in the statement has no product-one subsequence of length $4$. Suppose that $S = (\alpha^{t_1})^{[3]} \bdot (\alpha^{t_3}\tau)^{[3]} \bdot \alpha^{t_2}$, where $t_1, t_2, t_3 \in [0, 3]$ such that $t_1$ is even and $t_2$ is odd. Assume to the contrary that $S$ has a product-one subsequence $T$ of length $4$. Since $(\alpha^{t_1})^{[3]} \bdot \alpha^{t_2}$ is not product-one sequence, we must have $|T_{G_0}| = 2$, and thus either $T = (\alpha^{t_1})^{[2]} \bdot (\alpha^{t_3}\tau)^{[2]}$ or $T = \alpha^{t_1} \bdot \alpha^{t_2} \bdot (\alpha^{t_3}\tau)^{[2]}$. But they are nor product-one sequences. By the similar argument, we obtain that the remaining sequence has no product-one subsequence of length $4$.

\smallskip
(a) $\Rightarrow$ (b) Let $n \ge 2$ be even. Then $\mathsf s (G) = 4n$ by Theorem \ref{1.1}.2. Thus the assertion follows by Lemmas \ref{4.1} and \ref{2.4}.

\medskip
2. (b) $\Rightarrow$ (a) (1) The assertion follows from Lemma \ref{2.5}.

\smallskip
(2) By (1), it suffices to verify that a sequence $S$ with $|S_{G_0}| \ge 2$ given in the statement has no product-one subsequence of length $8$. Suppose that $S = (\alpha^{t_1})^{[3]} \bdot (\alpha^{t_3}\tau)^{[7]} \bdot \alpha^{t_2}$, where $t_1, t_2, t_3 \in [0,3]$ such that $t_1$ is even and $t_2$ is odd. Assume to the contrary that $S$ has a product-one subsequence $T$ of length $8$. Then we have either $|T_{G_0}| = 4$ or $|T_{G_0}| = 6$. If $|T_{G_0}| = 4$, then $T = (\alpha^{t_1})^{[3]} \bdot (\alpha^{t_3}\tau)^{[4]} \bdot \alpha^{t_2}$, but it is not product-one sequence. Thus $|T_{G_0}| = 6$, and hence we have either $T = (\alpha^{t_1})^{[2]} \bdot (\alpha^{t_3}\tau)^{[6]}$ or $T = \alpha^{t_1} \bdot \alpha^{t_2} \bdot (\alpha^{t_3}\tau)^{[6]}$. But they are not product-one sequences. By the similar argument, we obtain thet the remaining sequences have no product-one subsequence of length $8$.

\smallskip
(a) $\Rightarrow$ (b) Let $n \ge 2$. Note that $\mathsf E (G) = 6n$. Then the assertion follows by Lemmas \ref{4.2} and \ref{2.5}.
\end{proof}

%%%%%%%%%%%%%%%%%%%%%%%%%%%%%%%%%%%%%%%%%%%%%%%%%%%%%%%%%%%%%%%%%%%%%%%%%
%%                                      %%%%%%%%%%%%%%%
%%%%%%%%%%%%%%%%%%%%%%%%%%%%%%%%%%%%%%%%%%%%%%%%%%%%%%%%%%%%%%%%%%%%%%%%%

\providecommand{\bysame}{\leavevmode\hbox to3em{\hrulefill}\thinspace}
\providecommand{\MR}{\relax\ifhmode\unskip\space\fi MR }
\providecommand{\MRhref}[2]{%
  \href{http://www.ams.org/mathscinet-getitem?mr=#1}{#2}
}
\providecommand{\href}[2]{#2}

\end{document}